\documentclass[11pt,a4paper]{article}
\usepackage{amsmath}
\usepackage{amssymb}
\usepackage{epsfig}
\usepackage{latexsym,cite}
\usepackage[mathcal]{eucal}
\usepackage{lscape}
\usepackage{comment}
\usepackage{color}
\usepackage{hyperref}
\usepackage{graphicx}
\usepackage{caption}
\usepackage{subcaption} 
\usepackage{titlesec}
\usepackage{todonotes}
\usepackage{atveryend}
\makeatletter
\let\origcitation\citation
\AtEndDocument{\def\mycites{\@gobble}%
  \def\citation#1{\g@addto@macro\mycites{,#1}\origcitation{#1}}}
\AtVeryEndDocument{\typeout{***^^JCited keys: \mycites^^J***}}
\makeatother

\usepackage{comment}
\hypersetup{
    colorlinks=true,
    linkcolor=blue,
    filecolor=magenta,
    urlcolor=cyan,
    citecolor=red,
}
\newcommand{\dom}{{\rm dom\,}}
\newcommand{\epi}{{\rm epi\,}}

\newcommand{\ds}{\displaystyle}

\newcommand{\nexto}{\kern -0.54em}
\newcommand{\R}{\mathbb{R}}
\newcommand{\dN}{\mathbb{N}}

\newcommand{\la}{\langle}
\newcommand{\ra}{\rangle}

\newcommand{\proofbox}{\hspace{\fill}{$\Box$}}
\newtheorem{algorithm}{Algorithm}[section]
\newtheorem{lemma}{Lemma}[section]
\newtheorem{theorem}{Theorem}[section]
\newtheorem{corollary}{Corollary}[section]
\newtheorem{proposition}{Proposition}[section]

\newtheorem{fact}{Fact}[section]
\newtheorem{remark}{Remark}[section]

\newtheorem{definition}{Definition}[section]
\newenvironment{proof}{Proof.}{\proofbox}

\newcommand{\argmin}{\operatornamewithlimits{argmin}}

\begin{document}

\title{\bf An Inexact Deflected Subgradient Algorithm in Infinite Dimensional spaces}

\author{
Regina S. Burachik\thanks{Mathematics, UniSA STEM, University of South Australia, Mawson Lakes, S.A. 5095, Australia. Emails:~regina.burachik@unisa.edu.au, xuemei.liu@mymail.unisa.edu.au.}
\qquad
Xuemei Liu\footnotemark[1]}

\maketitle

\begin{quotation}
\begin{abstract}
\noindent We propose a duality scheme for solving constrained nonsmooth and nonconvex optimization problems in a reflexive Banach space. We establish strong duality for a very general type of augmented Lagrangian, in which we assume a less restrictive type of coercivity on the augmenting function. We solve the dual problem (in a Hilbert space) using a deflected subgradient method via this general augmented Lagrangian. We provide two choices of step-size for the method. For both choices, we prove that every weak accumulation point of the primal sequence is a primal solution. We also prove strong convergence of the dual sequence.
\end{abstract}
\end{quotation}

\begin{verse}
{\em Key words}\/: {\sf Augmented Lagrangian; Banach space; Nonconvex optimization; Nonsmooth optimization; Subgradient methods; Duality scheme; Penalty function methods.}
\end{verse}

\begin{verse} 
{\bf AMS subject classifications.} {\sf 49M29;  90C25; 90C26; 90C46; 65K10.}
\end{verse}

\pagestyle{myheadings}
\markboth{}{\sf\scriptsize Deflected Subgradient Algorithm and Optimal Control \ \ by R. S. Burachik and X. Liu}

\section{Introduction}

\noindent The (generalized) augmented Lagrangian duality theory is a powerful tool for solving nonconvex constrained optimization problems. Instead of tackling directly the constrained (primal) problem, we can recover the primal solution by solving the dual problem. This is particularly useful when the dual problem is easier to solve than the primal one. This is the case in the augmented Lagrangian duality framework. The dual problem is obtained from the Lagrangian function, which is a function that incorporates both the objective function and the information on the constraints. \emph {Strong duality\index{strong duality}} (i.e., when the primal and dual problems have the same optimal value) is a basic requirement when using a duality framework. For nonconvex problems, however, a positive gap may exist between the primal and dual optimal values when the classical Lagrangian is used. The augmented Lagrangian duality~\cite{RockWets,RYlagrType}, on the other hand, will have zero duality gap even in the nonconvex case, and will allow us to recover the solutions of the original (nonconvex) problem. We describe next the origins and up-to-date development of augmented Lagrangians.

\noindent The {\em linear augmented Lagrangian} as introduced in \cite[Chapter 11]{RockWets} is the sum of the classical Lagrangian and an augmenting term; in other words, it is the sum of the objective function, a linear term and an augmenting term. The {\em sharp Lagrangian}, introduced in \cite[Example 11.58]{RockWets}, is a linear augmented Lagrangian which adds to the classical linear term any norm function. 
The theory of Lagrangian duality is an active area of research, see, e.g., \cite{Bur2011,B2017,BKpenalty,BKpenaltypara,BYZ2017,HY2003,HY2005,NO2008,WYY2012,ZY2008,ZY2004,ZY2006,ZY2009,ZY2012}. In particular, \cite{ZY2006} and \cite{ZY2009} are for infinite dimensional settings. Dolgopolik in \cite{Dol2018} and \cite{Dol20182}  gives an excellent introduction on different types of Lagrangians and their applications in solving various kinds of problems. More general types of Lagrangian function have been studied in \cite{BY2011} and more recently in \cite{BKP2021}.


Our aim is to provide a primal--dual framework for the infinite dimensional setting using a general, although simple enough, Lagrangian. Our duality scheme is paired with an algorithmic framework, the Deflected Subgradient Method (DSG). Several works exist that use DSG algorithm within a similar primal--dual framework. Hence we make a comparison among these works and ours in terms of the range of applications (in finite or infinite dimensions), the Lagrangian, and the convergence results. Since some level of dual convergence is studied in all these works, we will rather focus on primal convergence results. 

Gasimov~\cite{Gdsg} proposed a deflected subgradient algorithm which uses the sharp Lagrangian, where the augmenting function is a norm, to solve finite dimensional optimization problems. This method has the desirable property that it generates a dual sequence with a strict improvement of the dual values in each iteration (dual strict improvement). It uses a Polyak-type step-size which requires the knowledge of the optimal dual value. One should note however that this knowledge can be difficult to obtain in practice, especially for non-convex problems. The analysis in Gasimov~\cite{Gdsg} establishes only convergence of the sequence of dual values to the optimal value. It goes without saying that  primal convergence is probably the most important feature of any primal--dual scheme, but unfortunately this is not studied in \cite{Gdsg}. Indeed, the example given by Burachik et al.~in~\cite[Example 1]{BGKIdsg} shows that the primal sequence in~\cite{Gdsg} may not converge to the primal solution. 

Later on, using the same primal--dual framework as~\cite{Gdsg},  the works \cite{BGKIdsg,BKMinexact2010,BIMdsg} developed further results on the step-size and convergence results.  In~\cite{BGKIdsg}, Burachik, Gasimov, Ismayilova and Kaya, establish the convergence of an \emph{auxiliary primal sequence} for the Polyak-type step-sizes.  In~\cite{BIMdsg} Burachik, Iusem and Melo propose an \emph{inexact version} of the DSG algorithm and prove \emph{auxiliary primal convergence} for inexact iterations. Burachik, Kaya and Mammadov~\cite{BKMinexact2010} devise an inexact version of the methods in~\cite{BGKIdsg,Gdsg} and show that the same convergence properties can be preserved when there is a level of inexactness in the solution of the subproblems. 

Burachik, Iusem and Melo~\cite{BIMdsg} use the sharp Lagrangian, and propose two choices of step-sizes which are independent of the optimal value. They establish primal convergence with these step-sizes, even for the case when the dual solution set is empty. Hence we will adopt these types of step-sizes for our work to inherit the nice primal convergence properties for our algorithm. 

These four works above, namely \cite{Gdsg,BGKIdsg,BIMdsg,BKMinexact2010}, are for the finite dimensional setting. Let us now recall the ones that apply to infinite dimensions, which will constitute the main motivation of the present paper.

Burachik, Iusem and Melo~\cite{BIMinexact2013} use a general type of augmented Lagrangian, which includes the Lagrangians used by the four previous works as particular cases. They extend the analysis in~\cite{BIMdsg} to the infinite dimensional setting. Namely, the primal problems are defined in a reflexive Banach space and the constraint functions are defined in a Hilbert space. They also use an inexact version of the DSG algorithm as in~\cite{BIMdsg}, and establish both the primal and dual convergence by adopting the types of step-sizes in~\cite{BIMdsg}.
 
Burachik and Kaya \cite{BKdsg}, and later Burachik, Freire and Kaya \cite{BFK2014}, incorporated a scaling symmetric matrix $A$ in the linear term of the Lagrangian in finite dimensions. This general type of augmented Lagrangian will be the focus of the present paper, since it provides a level of generality that allows, e.g., the full theoretical analysis of the penalty case, i.e., when $A$ is taken as the zero matrix.

The type of Lagrangian we focus on is an extension of the one in \cite{BFK2014,BKdsg} to infinite dimensions. It is associated with the following infinite dimensional equality constrained problem:
\begin{equation}\label{eq:generalPr}
    \min_{x\in X} \varphi(x)\;\;\;\; {\rm  s.t. } \;\;\;\; h(x)=0\,,
\end{equation}
where $X$ is a reflexive Banach space, $\varphi:X\to\R\cup\{\infty\}$ is lower semi-continuous;  and $h:X\to\R^m$ is continuous. The Lagrangian $l:X\times\R^m\times\R_+\to\R$ is defined as 
\begin{equation}\label{eq:generalLagrn}
    l(x,y,c):=\varphi(x)-\langle Ay,h(x)\rangle+ c\,\sigma(h(x))\,,
\end{equation}
where $x\in X$, $y\in\R^m$, $c\in [0,\infty)$, $A:\R^m\to \R^m$\!,\! is a continuous map, and $\sigma:\R^m\to\R_+$ verifies $\sigma(x)=0$ if and only if $x=0$ (see Definition \ref{def:basic} for more details). 

The sharp Lagrangian is a particular case of the general Lagrangian defined in~\eqref{eq:generalLagrn} when $A$ is the identity map and the penalty function $\sigma$ is any norm in $\R^m$. 
Moreover, the linear term reduces to the classical penalty function when $A$ is zero ~\cite{BKdsg,CC2010}. The numerical experiments in \cite{BKdsg} demonstrate that choosing suitable $A$ or $\sigma$ for various classes of problems can improve the computational performance.

In~\cite{BKdsg}, the exact version of the DSG algorithm is utilized and both auxiliary and primal convergence results for the Polyak-type step-size are obtained. Our analysis is inspired by~\cite{BKMinexact2010} in designing our inexact DSG algorithm and establishing the convergence results. Our results extend those in \cite{BKMinexact2010} in the following ways. 
\begin{itemize}
    \item[(i)] Our Lagrangian \eqref{eq:generalLagrn} has a general map $A$, which includes the identity and the zero map as particular cases. The case $A$ the identity matrix is studied in \cite{BKMinexact2010}, while the case $A=0$ is new. This opens the way for the analysis and implementation of penalty methods.
    \item[(ii)]  We establish strong duality for our very general type of Lagrangian. In particular, the function $\sigma$ we consider may not be coercive (see Definition \ref{def:assumptions_sigma}(a') and Theorem \ref{th:StrongDual}).
\end{itemize}

Regarding the study of the theoretical properties of our primal-dual setting, we point out that the proof of strong duality provided in \cite{BR2007} would cover our case. However, we provide our own proof here, because our type of Lagrangian allows us to provide a result which requires weaker assumptions (see more details in the paragraph right before Lemma~\ref{lem:seq_Lagrn}).

 In Section~\ref{ssec:PD} we show that our infinite dimensional framework (i) has no duality gap, and (ii) has a dual problem which is convex hence we can solve it using the techniques in convex analysis \cite{Pol1997,Ber1999,Nes2004}. In particular, the algorithm we introduce here can be seen as an epsilon-subgradient algorithm applied to the maximization of the dual function (see Remark \ref{rem:eps-sub}). 

\noindent  The paper is organized as follows. In Section~\ref{sec:Pre}, we give the preliminaries mainly on functional analysis, which help in building our primal--dual framework and establishing our convergence results. In Section~\ref{ssec:PD}, we give our primal--dual framework and its important assumptions. In particular, we will respond to the two questions above in this section, where we show the properties of this framework and give the related proofs. In Section~\ref{sec:DSG}, we state the DSG algorithm. We provide two choices of the step-size and establish the convergence results for both of the choices. Our conclusion is given in Section \ref{sec:conclusion}.

\section{Preliminaries}\label{sec:Pre}

We provide in this chapter some functional analysis tools for future use. Most of these results can be found in the text books of functional analysis such as those by Brezis~\cite{Brezis} and Kreyszig~\cite{Kreyszig}. We use Brezis's book on Functional Analysis~\cite{Brezis} as our main reference, and provide our own proof for results which are either not included in \cite{Brezis} or hard to track down elsewhere. The results we list here will be used in proving the properties of the primal--dual framework, as well as in establishing convergence results of the DSG algorithm.

Let $X$ be a reflexive Banach space, $X^*$ its topological dual (i.e., the set all continuous linear functionals from $X$ to $\R$), and $H$ a Hilbert space. 

We denote by $\langle \cdot, \cdot \rangle$ both the duality product in $X\times X^*$ and the scalar product in $H$. We denote by $\|\cdot\|$ the norm, where the same notation will be used for the norm both in $X$ and $H$.  We use the notation $\R_{++}$ for the positive real numbers,  $\R_{+\infty}:=\R\cup\{+\infty\}$ (sometimes $\R_{\infty}$ for short) and $\overline{\R}:=\R_{+\infty}\cup\{-\infty\}$. Given a function $g:X\to \overline{\R}$, define the {\em effective domain\index{effective domain}} of $g$ as $\dom g:=\{x\in X\::\: g(x)< +\infty\}$. We say that $g$ is {\em proper}\index{proper} if $g(x)>-\infty$ and $\dom g\neq \emptyset$.  Recall that the set $\epi g :=\{(x,t)\in X\times \R\::\: g(x)\le t\}$ is the {\em epigraph of}\index{epigraph} $g$, and that the set  $lev_{g}(\alpha):= \{x\in X\::\: g(x)\le \alpha\}$ is the $\alpha$-{\em level set of}\index{level set} $g$. Let $Y$ be a Banach space and consider a map $F:X\to Y$, the {\em graph of $F$}\index{graph} is the set $G(F):=\{(x,v)\in X\times Y\::\: v=F(x)\}$. Given $C\subset X$, the {\em indicator function of $C$}\index{indicator function} is defined as $\delta_C(v):=0$ if $v\in C$ and $+\infty$ otherwise. If $C=\{z\}$ is a singleton, we denote $\delta_{\{z\}}=:\delta_z$.

\subsection{Functional Analysis Tools}

 The topology induced by the norm (in $X$ or $H$), is called the {\em strong} topology. The weak topology in $X$ (weak topology in $H$) is the coarsest topology that makes all elements  of $X^*$ (all elements of $H^*=H$)  continuous.

 \begin{definition}[definitions related with the weak topology]\hfill\break
 \label{def:w-closed set}
 Let $X$ be a Banach space, and $H$ a Hilbert space. 
 \begin{itemize}
     \item[(a)] Let $K\subset X$. We say that $K$ is {\em weakly closed in $X$}\index{weakly closed} when it is closed w.r.t. the weak topology in $X$. 
     \item[(b)] We say that a function $h:X\to H$ is {\em weak-weak continuous}\index{weak-weak continuous} when $h^{-1}(U)\subset X$ is weakly open in $X$ for every $U\subset H$ weakly open in $H$.
      \item[(c)] We say that a function $h:X\to H$ is {\em weak-strong continuous}\index{weak-strong continuous} when $h^{-1}(U)\subset X$ is weakly open in $X$ for every $U\subset H$ strongly open in $H$. Strong-strong\index{strong-strong continuity} and strong-weak continuity\index{strong-weak continuity} are defined similarly.
      \item[(d)] We say that a function $\varphi:X\to\R_{+\infty}$ is {\emph {weakly lower semi-continuous}} \index{weakly lower semi-continuous} (\emph {w-lsc}) when it is lsc w.r.t. the weak topology in $X$. Namely, when $epi\; \varphi$ is w-closed.

 \end{itemize}
 \end{definition}

 We recall next some well-known facts from functional analysis.
 
  \medskip
 
  \begin{fact}\label{lem:Weakly Closed}
 Let $X$ be a Banach space, $H$ be a Hilbert space. Assume that $K\subset X$ is nonempty. The following hold.
     If $K\subset X$ is weakly compact, then it is weakly closed.
 \end{fact}
   \medskip
 
 In most of what follows, when a topological property is mentioned by its own, this means that the property holds w.r.t. the strong (i.e., the norm) topology. For instance, if we write ``$A$ is closed", we mean ``$A$ is {\sc strongly} closed". If a property holds w.r.t. the weak topology, we will mention the term ``weak" (or ``weakly") explicitly (e.g., weakly closed, weakly compact, etc.). 
  \medskip
  
 It is well-known that, in any metric space, compactness is equivalent to sequential-compactness. To clarify what the situation is for the  case of the weak topology in a Banach space $X$, we recall the following definitions. 
 
 \begin{definition}[weak compactness; sequential compactness; coercive] \hfill\break 
 \label{def:weakly_compact_coercive}
Let $X$ be a Banach space, $A\subset X$ and $\varphi : X\to \R_{+\infty}$. 
\begin{itemize}
    \item[(a)]  The set $A$ is {\em weakly-compact}\index{weakly-compact} when its weak closure, denoted as ${\overline{A}}^w$,  is compact w.r.t the weak topology.
    \item[(b)] A set $A\subset X$ is {\em sequentially-compact}\index{sequentially-compact} (respectively, {\em weakly sequentially-compact}) when every sequence $\{x_n\}\subset A$ has a subsequence  converging strongly (respectively, weakly) to a limit in $A$. 
     \item[(c)] The function $\varphi : X\to \R_{+\infty}$ is {\em coercive}\index{coercive} when $\lim_{\|x\|\to \infty} \varphi(x)=+\infty$. 
\end{itemize}
\end{definition}

 \medskip
 
The equivalence between compactness and sequential-compactness in normed spaces allows the use of sequences when dealing with compact sets in $X$. To be able to deal with weakly compact sets in $X$ in terms of sequences, we recall the following classical well-known result, which is \cite[Problem 10(3), p. 448]{Brezis}.

\begin{theorem}[Eberlein-Smulian]\label{th:ES}
Let $E$ be a Banach space and let $A\subset E$. Set $B:={\overline{A}}^w$ (i.e., $B$ is the weak closure of $A$). The following statements are equivalent.
\begin{itemize}
    \item[(i)] $B$ is weakly compact.
    \item[(ii)] $B$ is weakly sequentially-compact.
\end{itemize}
\end{theorem}

Next we quote results that connect boundedness, closedness and compactness both in strong and weak topologies. The next result, a corollary of Bourbaki-Alaoglu's theorem, is \cite[Corollary 3.22]{Brezis}. This result is a consequence of a separation result for convex sets, together with Bourbaki-Alaoglu's theorem.

\begin{theorem}\label{th:BA1}
Let $E$ be a reflexive Banach space. Let $K\subset  E$ be a bounded,
closed, and convex subset of $E$. Then $K$ is weakly compact.
\end{theorem}

\begin{corollary}\label{cor:CCB}
If $X$ is a Banach space, then every weakly compact set is closed and bounded.
\end{corollary}
\medskip

We quote next a result on sequential compactness that holds in reflexive Banach spaces.

\begin{theorem}\label{th:3.18}
Assume that $X$ is a reflexive Banach space and let $\{x_n\}$ be a bounded sequence in $X$. Then there exists a subsequence $\{x_{n_k}\}\subset \{x_n\}$ that converges in the weak topology.
\end{theorem}



In our analysis, we will consider level sets of w-lsc functions. For future use, we prove below a property that directly follows from the results quoted above. This property is well-known, but hard to track down as stated below. So we provide the proof here for convenience of the reader.

\begin{corollary}\label{cor:coer}
Let $X$ be a reflexive Banach space and assume that $\varphi: X\to \R_{+\infty}$ is w-lsc. The function $\varphi$ is coercive if and only if all its level sets are weakly compact. In this situation, all the level sets are closed and bounded.
\end{corollary}
\begin{proof}
Assume first that $\varphi$ is coercive and fix $\alpha\in \R$. We need to show that $lev_{\varphi}(\alpha)$ is weakly compact. By Eberlein-Smulian theorem, which is Theorem \ref{th:ES}, it is enough to show that $lev_{\varphi}(\alpha)$ is weakly sequentially compact. This means that every sequence in $lev_{\varphi}(\alpha)$ contains a subsequence weakly convergent to a limit, and this limit belongs to $lev_{\varphi}(\alpha)$. Indeed, take a sequence $\{x_k\}\subset lev_{\varphi}(\alpha)$. Since $\varphi$ is coercive this sequence is bounded, and by Theorem \ref{th:3.18} there exists a subsequence $\{x_{n_k}\}$ converging weakly to some $x\in X$. Since $\varphi$ is w-lsc, we can write
\[
\varphi(x)\le \liminf_{k\to \infty} \varphi(x_{n_k})\le \alpha,
\]
where the last inequality holds because $x_{n_k}\in lev_{\varphi}(\alpha)$ for all $k$. This implies that $x\in lev_{\varphi}(\alpha)$. Hence, the level sets $lev_{\varphi}(\alpha)$
are weakly sequentially compact. By Eberlein-Smulian theorem, they are weakly compact. Conversely, assume that $lev_{\varphi}(\alpha)$ is weakly compact for every $\alpha\in \R$. To show that $\varphi$ is coercive, it is enough to show that $lev_{\varphi}(\alpha)$ is bounded. This follows directly from Corollary \ref{cor:CCB}, which states that any weakly compact set must be bounded. Hence every level set $lev_{\varphi}(\alpha)$ is bounded and $\varphi$ is coercive. The last statement in the corollary is a direct consequence of Corollary \ref{cor:CCB}.
\end{proof}

\medskip

The next result is crucial in establishing the well-definedness of the algorithm we will present in section \ref{sec:DSG}. The result we quote below is \cite[Proposition 3.1.15]{BI}.

\begin{theorem}\label{th:existence}
Let $E$ be any topological space, and let $\varphi : E\to \R_{+\infty}$ be a proper function which is lsc (w.r.t. the topology of $E$). If  $A\subset E$ is compact and such that $\dom \varphi \cap A\neq \emptyset$, then $\varphi$ is bounded below on $A$ and it attains its minimum on $A$, which is finite.
\end{theorem}

\section{Primal and Dual Problems}\label{ssec:PD}

\subsection{Theoretical Framework}
The primal--dual framework we present here extends the one studied in~\cite{BKdsg,BFK2014} to the infinite dimensional setting. A particular case of our duality framework is the sharp Lagrangian in finite dimensions, as studied in~\cite{BGKIdsg,BKMinexact2010,BIMdsg,Gdsg}. The sharp Lagrangian has as augmenting term 
any norm, which motivates the terminology ``sharp". In infinite dimensions, Burachik, Iusem and Melo~\cite{BIMinexact2013} propose a related Lagrangian framework. The Lagrangian proposed in \cite{BIMinexact2013} uses a penalty function $\sigma(\cdot)$ with the same properties we study here, but 
the difference with our type of Lagrangian is in the linear term. Namely, our Lagrangian includes a map $A$ in the linear term, opening the way to the consideration of penalty methods for the particular case in which $A=0$. The framework in \cite{BIMinexact2013} deals with the case in which $A$ equals the identity map. Even though some of our proofs are similar to those in \cite{BIMinexact2013}, extra care is needed due to the presence of a general map~$A$. Since the identity map satisfies all the assumptions we make on the map $A$, all the results in \cite{BIMinexact2013} can be deduced from our analysis. We use a different and more involved method of proof 
for obtaining the main result in this chapter, namely, the strong duality property.

Let $X$ be a reflexive Banach space, and $\varphi:X\to \R_{\infty}$ be a proper function.  We consider the primal optimization problem
\begin{equation*}
(P)\qquad{\rm min}\;\;\!\varphi(x)\;\; {\rm  s.t. }\,\,x\ \text{in }X.
\end{equation*}

Following \cite[Section 2.2]{BR2007}, we embed problem $(P)$ into a family of parametrized problems by means of a function that coincides with the objective function when the parameter is zero. The tool we use is defined next.

\begin{definition}\label{def:PF}
A {\em dualizing parameterization} for $(P)$ is a
function $f:X\times H\to \bar{\R}$ that verifies $f(x,0)=\varphi(x)$ for all $x\in X$.  The {\em perturbation function} induced by this dualizing parameterization is defined as
$\beta :H\to \bar{\R}$ such that
\begin{equation}\label{prel:eq1}
\beta(z):=\inf_{x\in X} f(x,z).
\end{equation}
\end{definition}

The next definition, which is \cite[Definition 5.1]{BR2007}, will be a basic assumption for the dualizing parametrizations. It uses the concepts of weakly open and weakly compact sets. We recalled the latter concept in Definition \ref{def:weakly_compact_coercive}(a). Recall also that a set is weakly open set when its complement is weakly closed (see Definition \ref{def:w-closed set}(a)).

\medskip

\begin{definition}\label{level-f}
A function $f:X\times H\rightarrow \bar{\R}$ is said to be {\em weakly level-compact} if for each $\bar{z}\in H$ and $\alpha \in \R$
there exist a weakly open neighbourhood $U \subset H$ of $\bar{z}$, and a weakly compact set $B\subset X$, such that
\[ lev_{z,f}(\alpha):=\{x \in X: f(x,z)\leq \alpha \} \subset B\;\; \mbox{for all}\; z \in U.\]
In other words, there exist sets $U\subset H$ weakly open and $B\subset X$ weakly compact, such that $\bar{z} \in U$ and we have
\[ \bigcup_{z\in U} lev_{z,f}(\alpha):=\{x \in X: f(x,z)\leq \alpha\,\,\, \forall z\in U\} \subset B.\]
\end{definition}

If the duality parameterization is weakly-level compact, the corresponding perturbation function is {\em sequentially weak-lsc}.  Before establishing this fact, we recall next the definition.

\begin{definition}[sequentially weak-lsc function] \hfill\break 
\label{def:seq-w-lsc}
Let $\theta :H\rightarrow \R_{+\infty}$. We say that $\theta$
is {\em sequentially weakly lsc} if the following property holds.
\[
\hbox{If }u_n\rightharpoonup u, \hbox{ then } \theta(u)\le \liminf_{n\to \infty} \theta(u_n).
\]
\end{definition}

\begin{remark}[weak-lsc vs. sequentially weak-lsc]\hfill\break
\label{rem:w_lsc}
\rm In finite dimensions, or more generally in any metric space, there is no difference between semicontinuity and its sequential version. In an infinite dimensional Hilbert space, however, weak lsc as given in Definition \ref{def:w-closed set}(d) is more restrictive than its sequential version. In the former, the liminf inequality in Definition 
\ref{def:seq-w-lsc} must hold for any net weakly converging to a limit. Since a sequence is a particular case of a net, weak lsc implies sequential weak-lsc, and the converse, in general, does not hold. Indeed, while Definition \ref{def:w-closed set}(d) corresponds to weak closedness of the epigraph, Definition \ref{def:seq-w-lsc} corresponds to the latter set merely being sequentially weakly closed. 
\end{remark}

We will use the following type of functions for constructing our Lagrangian function.

\begin{definition}[augmenting function] \hfill\break 
\label{def:assumptions_sigma}
A function $\sigma :H\rightarrow \R_{+\infty}$ is an {\em augmenting function} if the following properties hold.
\begin{itemize}
    \item[(a)] The function $\sigma$ is proper, w-lsc and coercive (see Definitions \ref{def:w-closed set}(d) and \ref{def:weakly_compact_coercive}(c)).
    \item[(a')] The function $\sigma$ is proper, w-lsc, and satisfies the following condition: There exists $K_{\sigma}>0$ s.t. the set
    \[
    lev_{\sigma}(K_{\sigma})=\{z\in H\::\: \sigma(z)\le K_{\sigma}\},
    \]
    is bounded. We call this type of $\sigma$ {\em conditionally coercive}.
    \item[(b)] It holds that $\sigma(0)=0$ and $\ds\argmin_z\sigma(z)=\{0\}$.
\end{itemize}
\end{definition}

In what follows, we always assume that the function $\sigma$ used in the Lagrangian satisfies the assumptions of Definition \ref{def:assumptions_sigma}, either with (a) or (a'). Note that condition (a) (coercivity), is strictly stronger than (a') (conditional coercivity). If we are able to relax the requirements on $\sigma$ and just require conditional coercivity for $\sigma$, we will make it clear in our proofs. Otherwise, we may simply say that $\sigma$ is as in Definition \ref{def:assumptions_sigma}(a).  It is the less restrictive assumption on $\sigma$ the one we will use in our proof of strong duality. Before doing this, we proceed to establish the announced sequential w-lsc for the perturbation function. 

\begin{proposition}[sequential w-lsc of $\beta$]\hfill\break
\label{Prop:lsc_perturb}
Let $f:X \times Z \to \overline{\R}$ be weakly lower semicontinuous and weakly level-compact. Then the function $\beta$ defined by \eqref{prel:eq1} is sequentially-weakly lower semicontinuous.
\end{proposition}
\begin{proof}
Assume that $\beta$ is not sequentially weakly lower semicontinuous. This implies that there is a point $u$, a sequence $\{u_n\}_{n\in \dN}$, and $\varepsilon>0$ such that
\begin{itemize}
 \item[(i)] $u_n\rightharpoonup u$,
  \item[(ii)]$ \liminf_n \beta(u_n) < \beta(u) - \varepsilon.$
\end{itemize}
By weak-level compactness of $f$, there exists a weak neighborhood $W$ of $u$ such that the set
\[
\{x\in X\::\: f(x,z)\le \beta(u) - \varepsilon\}\subset B \text{ for all } z\in W,
\]
where $B$ is weakly compact in $X$. By (i), there exists $n_0\in \dN$ such that $u_n\in W$ for all $n\ge n_0$.  Therefore, for all $n\ge n_0$ we have that
\[
\{x\in X\::\: f(x,u_n)\le \beta(u) - \varepsilon\}\subset B.
\]
Calling $\tilde F(x):=\sup_{n\ge n_0} f(x,u_n)$, this implies that
\begin{equation}\label{set wc}
L:=lev_{\tilde F}(\beta(u) - \varepsilon)=\{x\in X\::\: \tilde F(x)\le \beta(u) - \varepsilon\}\subset B.
\end{equation}

Since $f(\cdot,v)$ is w-lsc for all $v\in H$, we deduce that $\tilde F$ is w-lsc too. Hence, the set $L$ is weakly closed. Since $L$ is a weakly closed subset of a weakly compact set, it is weakly compact.  We can apply now Eberlein-Smulian Theorem (Theorem \ref{th:ES}), to deduce that $L$ is weakly sequentially compact. By (ii) and the definition of liminf we can write
\[
\beta(u) - \varepsilon> \sup_{k\in \dN} \inf_{n\ge k} \beta(u_n)\ge \inf_{n\ge n_0} \beta(u_n).
\]
Take now $n_1\ge n_0$ such that for all $k\ge n_1$ we have
\[
\beta(u) - \varepsilon- 1/k>\inf_{n\ge n_0} \beta(u_n)=\inf_{n\ge n_0}\inf_{x\in X} f(x,u_n),
\]
where we used the definition of $\beta$ in the equality. Define now $\tilde G(x):=\inf_{n\ge n_0} f(x,u_n)$, so the above expression becomes
\[
\beta(u) - \varepsilon- 1/k>\inf_{n\ge n_0} \beta(u_n)=\inf_{x\in X} \inf_{n\ge n_0} f(x,u_n)=\inf_{x\in X} \tilde G(x),
\]
which holds for all $k\ge n_1$. Fix now an index $k\ge n_1$. By definition of infimum we can find  $x_k\in X$ such that $\tilde G(x_k)= \inf_{n\ge n_0} f(x_k,u_n)< \beta(u) - \varepsilon- 1/k$. Using a similar argument again for this fixed $k\ge n_1$, we can find $n_k\ge n_0$ and $u_{n_k}$ s.t. $f(x_k,u_{n_k})< \beta(u) - \varepsilon- 1/k$. Doing this for every $k\ge n_1$ and using the fact that $u_{n_k}\in W$ for all $k$ we deduce that the obtained sequence $(x_k)$ is contained in the sequentially compact set $L$. Thus there exists a subsequence of $(x_k)$ which is weakly convergent to a limit $\hat x\in L$. For simplicity, we still denote this weakly convergent sequence  by $(x_k)$. So we can assume that $x_k\rightharpoonup \hat x\in L$. Since $(u_{n_k})\subset (u_n)$ we have that $u_{n_k} \rightharpoonup u$. By w-lsc of $f$ and the definition of $\beta$ we obtain
\[
\beta(u)\le f(\hat x,u)\le \liminf_k f(x_{k},u_{n_k})\le \liminf_k \beta(u) - \varepsilon- 1/k= \beta(u) - \varepsilon,
\]
a contradiction. Therefore, $\beta$ must be sequentially weakly lsc.
\end{proof}

\bigskip
We present next the basic assumptions we need for the map $A$ (one of these involves the augmenting function $\sigma$ as in Definition \ref{def:assumptions_sigma}).

Assume that map $A:H\to H$ verifies the following properties: 
\begin{itemize}
    \item [${\bf\hypertarget{A0}{(A_0)}}$]
            $\sigma(z)\geq \|A(z)\|$, for all $z\in H$.
     \item [${\bf\hypertarget{A1}{(A_1)}}$]
            For every $y\in H$, the function $\langle A(\cdot), y \rangle$ is w-usc (i.e., $-\langle A(\cdot), y \rangle$ is w-lsc for every $y\in H$).
           
\end{itemize}

\begin{remark}[Assumptions $\bf (A_0)$--$\bf (A_1)$]\hfill\break 
\label{rem:assumption_A}
\rm
\begin{itemize}
    \item[(i)]
 By Definition~\ref{def:assumptions_sigma} (b), ${\bf\hyperlink{A0}{(A_0)}}$ trivially implies that $A(0)=0$. This assumption will be used in Proposition~\ref{prop:dualProperty}(iii) to derive the weak duality property. Assumption ${\bf\hyperlink{A0}{(A_0)}}$ is also important in obtaining a non-decreasing property of the dual function, as we will see in Proposition~\ref{prop:dualProperty}(ii).
 \item[(ii)]
  ${\bf\hyperlink{A1}{(A_1)}}$ plays a role in obtaining the strong duality for our primal--dual framework in Lemma \ref{lem:seq_Lagrn} and Theorem~\ref{th:StrongDual}.
\end{itemize}
 \end{remark}


We next list more assumptions on our primal--dual framework.
\begin{itemize}
    \item[\hypertarget{H0}{(H0)}] The objective function $\varphi:
X\to \R_{+\infty}$
is proper and w-lsc. 
    \item[\hypertarget{H1}{(H1)}] The function $\varphi$ has weakly compact level sets.
    \item[\hypertarget{H2}{(H2)}] The dualizing parameterization $f$ is proper (i.e., $\dom f\neq \emptyset$ and $f(x,z)>-\infty,$ \,$\forall \, (x,z)\in X\times H$), w-lsc and weakly level-compact (see Definition \ref{level-f}). 
\end{itemize}\medskip

We define next the problem dual to $(P)$, via our augmented Lagrangian function.

\begin{definition}[augmented Lagrangian and associated dual problem] \hfill\break 
\label{def:basic}
With the notation of Problem $(P)$, let
\begin{itemize}
    \item[(a)] $f$ be a dualizing parameterization as in Definition \ref{def:PF}, satisfying assumption {\rm \hyperlink{H2}{(H2)}},
     \item[(b)] $A:H\to H$ be a function verifying the assumptions ${\bf\hyperlink{A0}{(A_0)}}$--${\bf\hyperlink{A2}{(A_1)}}$.
      \item[(c)] $\sigma$ be an augmenting function as in Definition \ref{def:assumptions_sigma}, with (a') instead of (a).
\end{itemize}
The {\em augmented Lagrangian} for Problem $(P)$ is defined as
\begin{equation}\label{eq:Lagrn}
\ell(x,y,c):=\inf_{z\in H}\{f(x,z)-\langle A(z),y\rangle + c\sigma(z)\}.
\end{equation}
The {\em dual function} $q:H\times \R_+\to\R_{-\infty}$ is defined as
\begin{equation}
\label{eq:dual_function}
q(y,c):=\inf_{x\in X} \ell(x,y,c).
\end{equation} 
The \emph{dual problem} of $(P)$ is given by 
\begin{equation*}\label{eq:D_probm}
(D) \qquad {\rm maximize}\;\;\! q(y,c) \;\; {\rm  s.t. }\,\, (y,c)\in H\times \R_+.
\end{equation*}
Denote by $\ds M_P:=\inf_{x\in X}\varphi(x)$ and by $\ds M_D:=\sup_{(y,c)\in
H\times \R_+}q(y,c)$ the optimal values of the primal and dual
problem, respectively.
The primal and dual solution sets are
denoted by $S(P)$ and $S(D)$, respectively.
\end{definition}
\medskip
\begin{remark}[finite primal value for Problem $(P)$] \hfill\break 
\rm
By definition of duality parameterization and assumption \hyperlink{H0}{(H0)}, we have that $\varphi$ is proper, so by Definition \ref{def:PF}, the following holds
\begin{equation}
    \label{eq:primal_perturb}
M_P=\beta(0)<+\infty. 
\end{equation}
\end{remark}

\subsection{Properties of the Primal--Dual Setting}\label{sec:PD properties}

We next present some basic properties of the dual function given in Definition \ref{def:basic}.

\begin{proposition}[properties of the dual function] \hfill\break 
\label{prop:dualProperty} Let $q$ be the dual function which is defined in \eqref{eq:Lagrn}. The following facts hold.
\begin{itemize}
    \item[(i)] The dual function  $q$ is concave and
weakly upper-semicontinuous {\rm (w-usc)}.
    \item[(ii)] If $c \geq c_1$ then $q(y,c) \geq q(y,c_1)$ for all $y \in H$. In particular, if
$(y,c_1)$ is a dual solution, then also $(y,c)$ is a dual solution for all $c\geq c_1$.
    \item[(iii)] Assume ${\bf\hyperlink{A0}{(A_0)}}$ holds. The weak duality property holds for the $(P)-(D)$ primal -- dual framework, i.e. 
    \[
   M_D= \sup_{(y,c)\in H\times \R_+ }q(y,c)\le \inf_{x\in X}\varphi(x)=M_P,
    \]
    where $\varphi$ verifies {\rm \hyperlink{H0}{(H0)}}, and $q$ is as in Definition~\ref{def:basic}. 
\end{itemize}
\end{proposition}
\medskip
\begin{proof} \noindent(i) We show that $q$ is w-usc and concave simultaneously. By \eqref{eq:dual_function}, $q$ is the infimum of a family of w-usc and concave functions. Indeed, define $\psi_{xz}:H\times\R\to\R$, as $\psi_{xz}(y,c):=f(x,z)-\langle A(z),y\rangle + c\sigma(z)$. Then $\psi_{xz}$ is w-continuous and concave  (actually affine), w.r.t. the variable $(y,c)$. Now the concavity and weak-upper semicontinuity of $q$ follow from \eqref{eq:dual_function}.

We now proceed to show (ii). The fact that $q(y,\cdot)$ is non-decreasing follows directly from the definition and the fact that $\sigma(z)\ge 0$. Let now $(y,c_1)\in S(D)$. So $q(y,c_1)=M_D$. For every $c\ge c_1$ we use the non-decreasing property to write
\[
M_D\ge q(y,c)\ge q(y,c_1)=M_D,
\]
so $q(y,c)=M_D$ for all $c\ge c_1$.

Let us now show (iii). Using equations~\eqref{eq:Lagrn} and \eqref{eq:dual_function}, we have that 
\begin{equation}\label{eq:weak_dual}
   \begin{array}{rcl}
M_D=\ds \sup_{(y,c)\in H\times \R_+ }q(y,c) &=&  \ds\sup_{(y,c)\in H\times \R_+ }\inf_{x\in X} \inf_{z\in H} \{f(x,z)-\langle A(z),y\rangle + c\sigma(z)\}\\ [3mm]

&=& \ds\sup_{(y,c)\in H\times \R_+ }\inf_{z\in H} \{\inf_{x\in X} f(x,z)\} -\langle A(z),y\rangle + c\sigma(z)\\ [3mm]
&=& \ds\sup_{(y,c)\in H\times \R_+ }\inf_{z\in H} \{\beta(z)-\langle A(z),y\rangle + c\sigma(z)\} \\ [3mm]
&\le & \ds\sup_{(y,c)\in H\times \R_+ }\{\beta(0)-\langle A(0),y\rangle + c\sigma(0)\}=\beta(0) \\ [3mm]
&&\\
&=& \ds \inf_{x\in X} \varphi(x) = M_P,
   \end{array}
\end{equation}
where we used the definition of $\beta$ (see Definition \ref{def:PF}) in the fourth equality. We also used the fact that $A(0)=0$ (which holds by $(A_0))$ and the fact that $\sigma(0)=0$  (see Remark~\ref{rem:assumption_A}(i) and Definition~\ref{def:assumptions_sigma}(b)).
\end{proof}
\medskip

We note that the result above only requires for $\sigma$ to verify property (b) in Definition~\ref{def:assumptions_sigma}.

\medskip 
We now proceed to establish the zero duality gap property for our primal dual setting. Burachik and Rubinov show strong duality for very general primal -- dual frameworks in \cite{BR2007}. In their analysis, they use \emph{abstract convexity}\index{abstract convexity} tools~\cite{Rub2000}. Even though we can deduce the zero duality gap property as a consequence of their analysis, we prefer to prove this fact directly here. We do this because one of the assumptions used in \cite{BR2007} can be relaxed in our setting. Namely, we can replace the assumption of w-lsc of $\beta$ (used in \cite{BR2007}) by just {\em sequential}-w-lsc of $\beta$. Recall that the latter property holds for our function $\beta$, as established in Proposition \ref{Prop:lsc_perturb}.
\medskip

\begin{lemma}[properties of the Lagrangian approximation]\hfill\break
\label{lem:seq_Lagrn}
Consider the primal problem (P) and its dual problem
(D). Assume that {\rm\hyperlink{H0}{(H0)}--\hyperlink{H2}{(H2)}}, ${\bf\hyperlink{A0}{(A_0)}}$ and ${\bf\hyperlink{A1}{(A_1)}}$ hold. Assume that the augmenting function $\sigma$ verifies Definition \ref{def:assumptions_sigma}(a')(b). Suppose that there exists some
$(\bar y,\bar c)\in H\times \R_+$ such that $q(\bar y,\bar c)>-\infty$. For $n\in \dN$, define $\gamma_n:H\to \R_{-\infty}$ as
\[
\gamma_n(z):=\beta(z)-\langle A(z),\bar y\rangle+n\sigma(z).
\]
There exists a sequence $(v_n)\subset H$ with the following properties.
\begin{itemize}
    \item[(i)] There exists $n_0\in \dN$  such that 
    \begin{equation}\label{eq:s1}
 \gamma_n(v_n)=\inf_{z\in H} \gamma_n(z),\,\,\forall \, n\ge n_0.
    \end{equation}
    \item[(ii)] The sequence $(v_n)\subset H$ verifying \eqref{eq:s1} is bounded and converges weakly to zero.
   \end{itemize}
\end{lemma}
\begin{proof} Take $(\bar y,\bar c)\in H\times \R_+$ as given in the assumption of the theorem. Let $\bar n:= [\bar c]+1$ where $[\cdot]$ denotes the integer part (or floor)  of a real number. Then $\bar n\in \dN$ and 
since $q(\bar y,\cdot)$ is increasing we have that $q(\bar y,\bar n)\ge q(\bar y,\bar c)>-\infty$. For any $n\in \dN$, denote by $s(n):=\inf_{z\in H} \gamma_n(z)$. Since $q(\bar y,\bar n)>-\infty$ there exists $r_0\in \R$ such that 
\[
r_0< q(\bar y,\bar n)= \inf_{z\in H} \beta(z) -\langle A(z),\bar y\rangle+\bar n\sigma(z)=s(\bar n),
\]
which, together with the fact that $(s(n))$ is an increasing sequence, yields $s(n)>r_0$ for all $n\ge \bar n$. From now on, we consider the sequence $(s(n))$ for $n\ge \bar n$. Observe that this sequence $(s(n))$ is monotone increasing, bounded below by $r_0$ and bounded above by $\beta(0)=M_P$. Indeed, the monotonicity property follows from Proposition \ref{prop:dualProperty}(ii). The statement on the upper bound follows directly from the definition of $s(\cdot)$. Namely, for every $n\ge \bar n$ we have that $s(n)\le \gamma_n(0)=\beta(0)$, where we used assumption ${\bf(A_0)}$ and the fact that $\sigma(0)=0$. Altogether, the latter properties imply that the sequence $(s(n))$ converges (increasingly) to a limit $\bar s\le \beta(0)$. 

\noindent Proof of (i). Fix $n\ge \bar n$, we have that
\[
s(n)+1/k >s(n)= \inf_{z\in H} \gamma_n(z),
\]
for every $k> \bar n$. The definition of infimum allows us to find $v^k_n\in H$ such that 
\begin{equation}\label{eq:s1a}
0\le \gamma_n(v^k_n) - s(n)<1/k,
\end{equation}
where we used the definition of $s(\cdot)$ in the first inequality. For every $k> \bar n$, we use the definition of $s(\bar n)$ to write
\begin{equation}\label{eq:s1b}
s(\bar n)\le \gamma_{\bar n}(v^k_n)= \gamma_{k}(v^k_n)+(\bar n -k) \sigma(v^k_n),
\end{equation}
which re-arranges as
\begin{equation}\label{eq:s1e}
(k-\bar n)\sigma(v^k_n)\le \gamma_{k}(v^k_n)-s(\bar n).
\end{equation}
By \eqref{eq:s1a} and the established properties of the sequence $(s(n))$, we know that 
\begin{equation}\label{eq:s1d}
\gamma_{n}(v^k_n)<s(n)+1/k\le \bar s +1/k < \bar s +1,
\end{equation}
for all $k>\bar n$. Using \eqref{eq:s1d} in \eqref{eq:s1e}, and re-arranging the resulting expression we obtain
\begin{equation}\label{eq:s1c}
\sigma(v^k_n)\le  \dfrac{ (\bar s - s(\bar n)) +1 }{(k-\bar n)},
\end{equation}
where we also used the fact that $k> \bar n$. Take now $k_0:= \bar n + \dfrac{ (\bar s - s(\bar n)) +1 }{ K_{\sigma}}$, where $K_{\sigma}>0$ is as in Definition \ref{def:assumptions_sigma}(a'). Then it is direct to check that
\[
\sigma(v^k_n)\le  \dfrac{ (\bar s - s(\bar n)) +1 }{(k-\bar n)}<K_{\sigma},
\]
for all $k\ge k_0$. Now Definition \ref{def:assumptions_sigma}(a') implies that the set
\[
T:=\{v^k_n \::\: n\ge \bar n, k\ge k_0\},
\]
is bounded. In particular, for a fixed $n\ge \bar n$ the sequence $(v^k_n)_{k\ge k_0}$ is bounded and hence it has a subsequence that converges weakly to some $v_n\in H$. To keep notation simple, we still denote the weakly convergent subsequence by $(v^k_n)_{k\ge k_0}$. By Proposition \ref{Prop:lsc_perturb}, we know that $\beta$ is sequentially-w-lsc. By ${\bf\hypertarget{A1}{(A_1)}}$ and Definition \ref{def:assumptions_sigma}(a'), $-\langle A(\cdot),\bar y\rangle$ and $\sigma$ are w-lsc (and hence sequentially-w-lsc), we deduce that $\gamma_n$ is sequentially-w-lsc. Using the sequential w-lsc of $\gamma_n$ and the first inequality in
\eqref{eq:s1d} we can write
\[
\gamma_n(v_n)\le \liminf_{k\to \infty} \gamma_{k}(v^k_n)\le \liminf_{k\to \infty} s(n)+1/k=s(n)\le \gamma_n(z),
\]
for every $z\in H$ and every fixed $n\ge \bar n$. The last inequality in the expression above follows from the definition of $s(\cdot)$. Statement (i) now follows with $n_0:=\bar n$, by taking $z=v_n$ in the above expression.

\noindent Proof of (ii). Take now the sequence $(v_n)$ defined in part (i). Note that the set $T$ defined in part (i) is bounded, so there exists a closed ball $B_0$ such that $T\subset B_0$. By Theorem \ref{th:BA1}, $B_0$ is weakly compact (and by Fact \ref{lem:Weakly Closed} weakly closed). This implies that the weak closure of $T$ must be contained in $B_0$. Namely,
\[
{\overline{T}}^w\subset \overline{B_0}^w=B_0,
\]
showing that $\overline{T}^w$is bounded. By construction (see proof of (i)), every $v_n$ is a weak limit of a sequence in $T$, so we deduce that
\[
(v_n)\subset {\overline{T}}^w\subset B_0,
\]
showing that $(v_n)$ is bounded. Thus the boundedness statement in (ii) holds.  Let us proceed to show now that the sequence $(v_n)$ converges weakly to zero.  To prove this fact, we will show that every weakly convergent subsequence must converge to zero. If the latter is true, zero is the only weak accumulation point, so the whole sequence must weakly converge to zero. We have just established that $(v_n)$ is bounded, so by Theorem \ref{th:3.18}, it has weakly convergent subsequences. Take any such subsequence, denoted by $(v_{n_j})_{j\in \dN}$, converging weakly to some $v$. Since $(v_{n_j})\subset (v_n)$ and $n\ge \bar n$, we can take ${n_j}>\bar n$. Following the same steps as in \eqref{eq:s1b}-\eqref{eq:s1c} with $k={n_j}>\bar n$ and $v_{n_j}$ in place of $v_n^k$ we have that
\[
\sigma(v_{n_j})\le  \dfrac{ (\bar s - s(\bar n)) +1 }{(n_j-\bar n)},
\]
with $n_j\to \infty$. By the w-lsc of $\sigma$ we can write
\[
0\le \sigma(v)\le \liminf_{j\to \infty} \sigma(v_{n_j})\le \liminf_{j\to \infty}\frac{ (\bar s - s(\bar n)) +1 }{(n_j-\bar n)}=0,
\]
so $\sigma(v)=0$ and the assumptions on $\sigma$ yield $v=0$. This shows that every weak accumulation point of $(v_n)$ must be equal to zero, and hence the whole sequence converges weakly to zero, completing the proof of  (ii).


\end{proof}

We are now ready to establish the strong duality property of our primal-dual framework.

\begin{theorem}[strong duality for $(P)$--$(D)$ framework] \hfill\break 
\label{th:StrongDual}
Consider the primal problem (P) and its dual problem
(D). Assume that {\rm\hyperlink{H0}{(H0)}--\hyperlink{H2}{(H2)}}, ${\bf\hyperlink{A0}{(A_0)}}$ and ${\bf\hyperlink{A1}{(A_1)}}$ hold. Assume that the augmenting function $\sigma$ verifies Definition \ref{def:assumptions_sigma}(a')(b). Suppose that there exists some
$(\bar y,\bar c)\in H\times \R_+$ such that $q(\bar y,\bar c)>-\infty$. Then the zero-duality-gap property holds, i.e. $M_P=M_D$.
\end{theorem}
\begin{proof} Recall that weak duality (i.e., that $M_D\le M_P$) holds in our setting, as established earlier in Proposition~\ref{prop:dualProperty}(iii). Hence, we only need to prove that  $M_D\ge M_P$. By Lemma \ref{lem:seq_Lagrn}(i), we can take a sequence $(v_n)$ verifying \eqref{eq:s1}. Our first step is to show the following inequality. 
\begin{equation}\label{eq:MD1}
    M_D\ge \liminf_{n} \gamma_n(v_n),
\end{equation}
where $\gamma_n$ and $v_n$ are as in Lemma \ref{lem:seq_Lagrn}. Using $n_0$ as in Lemma \ref{lem:seq_Lagrn}(i) and  the definition of $M_D$, we have that
 \begin{equation}\label{eq:strong_dual1}
     \begin{array}{rcl}
          M_D &=& \ds\sup_{(y,c)\in H\times \R_+ }\inf_{z\in H} \{\beta(z)-\langle A(z),y\rangle + c\sigma(z)\} \\[3mm]
          &\ge& \ds\inf_{z\in H} \{\beta(z)-\langle A(z),\bar y\rangle + n\sigma(z)\} \\[3mm]
          &=& \ds \{\beta(v_n)-\langle A(v_n),\bar y\rangle + n\sigma(v_n)\}= \gamma_n(v_n),
     \end{array}
 \end{equation}
where we used the fixed choice of $(y,c):=(\bar y,n)$ with $n>n_0$ in the inequality, fact \eqref{eq:s1} in the second equality, and the definition of $\gamma_n$ in the last one. Inequality \eqref{eq:MD1} now follows by taking $\liminf$ in \eqref{eq:strong_dual1}. Using \eqref{eq:MD1}, the definition of $\gamma_n$ and the properties of $\liminf$ we deduce that
\begin{equation}
\label{eq:strong_dual2}
     \begin{array}{rcl}
   M_D   &\ge & \liminf_{n} \gamma_n(v_n)\\[3mm]
   &= &\liminf_{n}\beta(v_n)+ \,\left[-\langle A(v_n),\bar y\rangle + n\sigma(v_n)\right]\\[3mm]
     &\ge & \liminf_{n}\beta(v_n) + \liminf_{n} \left[ -\langle A(v_n),\bar y\rangle + n\sigma(v_n) \right]\\[3mm]
     &\ge &\liminf_{n}\beta(v_n) + \liminf_{n} \left[-\langle A(v_n),\bar y\rangle \right] \\[3mm]
     &\ge & \beta(0)+0=\beta(0)=M_P, 
  \end{array}
 \end{equation}
where we used the fact that $ n\sigma(v_n)\ge 0$ in the third inequality. In the last inequality we used Lemma \ref{lem:seq_Lagrn}(ii), namely the fact that $(v_n)$ converges weakly to zero and the fact that $A(0)=0$. More precisely, using the (sequential) w-lsc of the functions $\beta$ and $-\langle A(\cdot),\bar y\rangle$, we obtain
\[
\liminf_{n} \left[-\langle A(v_n),\bar y\rangle \right]\ge -\langle A(0),\bar y\rangle=0,
\]
and
\[
\liminf_{n}\beta(v_n)\ge \beta(0).
\]
Both facts were used in the last inequality of \eqref{eq:strong_dual2}. Since we already have that $M_D\le M_P$, we have thus established that $M_P=M_D$.
\end{proof}

\begin{definition}[superdifferential of a concave function]  \hfill\break
\label{rSupDiff}
Let $H$ be a Hilbert space and $g:H\to \R_{-\infty}$ be a concave function. Take $r\geq 0$. The $r$-superdifferential
of $g$ at $w_0\in {\rm dom}(g):=\{w\in H\::\: g(w)>-\infty\}$ is the set $\partial_r g(w_0)$ defined by
\[\partial_r g(w_0):=\{v\in H : g(w)\leq g(w_0) + \langle v,w-w_0 \rangle + r, \;\;\forall v \in H\}. \]
\end{definition}

\begin{definition}[approximations for the primal--dual and Lagrangian] \hfill\break
\label{def:aprrox_Lagn}
We say that
\begin{itemize}
    \item[(i)] $x_*\in X$ is an $\epsilon$-optimal primal solution\index{$\epsilon$-optimal primal solution} of $(P)$ if $\varphi(x_*)\leq M_P+\epsilon$
    \item[(ii)] $(y_*,c_*)\in H\times\R_+$ is an $\epsilon$-optimal dual solution\index{$\epsilon$-optimal primal solution} if $q(y_*,c_*)\geq M_D-\epsilon$.
    \item[(iii)] For $r\ge 0$ define the set
\begin{equation}\label{eq:aproxvalue}
X_r(y,c):=\{(x,z) \in X\times H:
f(x,z)-\langle A(z),y\rangle + c\sigma(z)
\leq q(y,c)+r\},
\end{equation}
which contains all  $r$-minimizers of the augmented Lagrangian.
\item[(iv)] Fix $(w,c)\in H\times\R_+$ and define $\Phi_{(w,c)}:X\times H\to\
\bar{\R}$ as
\begin{equation}\label{eq:ee1}
\Phi_{(w,c)}(x,z):=
f(x,z)-\langle A(z),w\rangle + c\sigma(z).
\end{equation}
\end{itemize}
\end{definition}

\begin{remark}[the dual set of the approximation for the Lagrangian]\hfill \break
\rm\label{step1a}
By definition of $q$ as an infimum, for every $r>0$ and every $(y,c)$ such that $q(y,c)>-\infty$, there exists $(x,z)$ such that $f(x,z)-\langle A(z),y\rangle + c\sigma(z)< q(y,c)+r$. Therefore, for every $r>0$ and every $(y,c)$ such that $q(y,c)>-\infty$, we have that $X_r(y,c)$ is nonempty.
\end{remark}

The result below extends \cite[Proposition 3.1, parts (i) and (iii)]{BIMinexact2013}, where the particular case in which $A=I$, the identity map in $H$, is considered. Since the proof follows, mutatis mutandis, the same steps as those in \cite[Proposition 3.1, parts (i) and (iii)]{BIMinexact2013}, we omit it.

\begin{proposition}\label{gradz0}
If $(\hat{x},\hat{z}) \in X_r(\hat{y},\hat{c})$, then the following facts hold.\\
\noindent i) For all $r\ge0$, $(-A(\hat{z}),\sigma(\hat{z})) \in \partial_rq(\hat{y},\hat{c})$.\\
\noindent ii) If $M_D\le M_P$ and $\hat z=0$, then $\hat x$ is a $r$-optimal primal solution, and $(\hat y,\hat c)$ is a $r$-optimal dual solution. In particular, if $r\le\epsilon$, then $\hat x$ is a $\epsilon$-optimal primal solution, and $(\hat y,\hat c)$ is a $\epsilon$-optimal dual solution.
\end{proposition}

\medskip
{\bf From now on we assume that the hypotheses of Theorem~\ref{th:StrongDual} are verified, and hence we have $M_P=M_D$.}
\medskip

The following result establishes several properties of the primal--dual solution sets, as well as compactness properties of the level sets of the function $\Phi_{(y,c)}$ defined in \eqref{eq:ee1}. The techniques of the proof for parts (i), (ii), the non-emptiness of the set in \eqref{eq:LS}, and (iiiA) are standard, and can be found in \cite[Lemma 3.1]{BIMinexact2013}.  Hence we will omit their proofs. The proof of part (iiiB), however, is new because the coercivity assumption on $\sigma$, which is used in \cite{BIMinexact2013}, is relaxed to the weaker version of Definition \ref{def:assumptions_sigma} with condition (a'). Hence we present here the proof of this part.

\medskip

\begin{theorem}[the compact level set of the Lagrangian]\hfill \break
\label{th:compact_level_set_Lagrangian}
Consider the primal problem (P) and its dual problem
(D). Suppose that {\rm\hyperlink{H0}{(H0)}--\hyperlink{H2}{(H2)}}, ${\bf\hyperlink{A0}{(A_0)}}$ and ${\bf\hyperlink{A1}{(A_1)}}$ hold.  The following statements hold.
\begin{itemize}
    \item[(i)] The set $S(P)\neq \emptyset$ and $M_P\in \R$.
    \item[(ii)] Let $({\hat y},\hat c)\in H\times \R_+$ be such that $q({\hat y},\hat c)>-\infty$ and consider the set 
    \[
    T:=\{ (w,c)\in H\times \R_+\::\: c>\hat c+\|w-{\hat y}\|\}.
    \]
    Then,
  \begin{itemize}
  \item[(iiA)] $T\subset \dom q$, i.e., $q(w,c)>-\infty$ for every $(w,c)\in T$. 
  \item[(iiB)] If $({\hat y},\hat c)\in S(D)$ then $T\subset S(D)$.
  \end{itemize}   
\item[(iii)] For every $s\ge M_P$, and every $(w,c)\in H\times \R_+$, the level set
\begin{equation}\label{eq:LS}
lev_{\Phi_{(w,c)}}(s)=\{(x,z)\in X\times H:\Phi_{(w,c)}(x,z)=
                f(x,z)-\langle A(z),w\rangle +c\sigma(z)\leq s\},
\end{equation}
is not empty. 
\begin{itemize}
  \item[(iiiA)] Let $({\hat y},\hat c)$ be as in (ii) and assume that $\sigma$ verifies Definition \ref{def:assumptions_sigma} with condition (a).  Then  the level set in \eqref{eq:LS}  is weakly-compact for every $(w,c)\in T$. In this situation, there exists
$(\tilde{x},\tilde{z})$ such that
\begin{equation}\label{Lem:eq2}
    q(w,c)=f(\tilde{x},\tilde{z}) -\langle A(\tilde{z}),w \rangle + c\sigma(\tilde{z}).
\end{equation}
\item[(iiiB)] Let $({\hat y},{\hat c})$ be as in (ii) and assume that $\sigma$ verifies Definition \ref{def:assumptions_sigma} with condition (a').  Define the set
 \[
    \tilde T(s):=\{ (w,c)\in H\times \R_+\::\: c> {\hat c}+ \left( \frac{s-q({\hat y},{\hat c})}{K_{\sigma}} \right)+\|w-{\hat y}\|\}\subset T,
    \]
where $K_{\sigma}>0$ is as in Definition \ref{def:assumptions_sigma}(a'). Then  the level set in \eqref{eq:LS} is weakly-compact for every $(w,c)\in \tilde T(s)$. In this situation, there exists
$(\tilde{x},\tilde{z})$ such that \eqref{Lem:eq2} holds.
\end{itemize}   
  \end{itemize}   
%

\end{theorem}
\begin{proof} The proof of parts (i), (ii), the non-emptiness of the set in \eqref{eq:LS}, and (iiiA) are similar to \cite[Lemma 3.1]{BIMinexact2013}.
We proceed to establish (iiiB). Note first that $\tilde T(s)\subset T$ because $\left( \frac{s-q({\hat y},{\hat c})}{K_{\sigma}} \right)\ge 0$. Indeed, note that $s\ge M_P=M_D\ge q({\hat y},{\hat c})$ and $K_{\sigma}>0$. It remains to show that $lev_{\Phi_{(w,c)}}(s)$ is weakly compact under the assumptions given in (iiiB). Namely, we need to show that the level set in \eqref{eq:LS} is weakly compact for every $(w,c)\in \tilde T(s)$. By Theorem \ref{th:ES}, it is enough to show that the set $lev_{\Phi_{(w,c)}}(s)$ is weekly sequentially compact. The latter means that every sequence contained in $lev_{\Phi_{(w,c)}}(s)$ has a weakly convergent subsequence, and that the limit of the weakly convergent subsequence belongs to $lev_{\Phi_{(w,c)}}(s)$. Take a sequence $\{(x_k,z_k)\}\subset lev_{\Phi_{(w,c)}}(s)$. We start by showing that $\{z_k\}$ has a weakly convergent subsequence. Indeed, by definition of $lev_{\Phi_{(w,c)}}(s)$ we have
\[
\begin{array}{rcl}
 s &\geq& f(x_k,z_k)-\langle A(z_k),w \rangle + c\sigma(z_k) \\[2.mm]
&=&
f(x_k,z_k)-\langle A(z_k),{\hat y} \rangle + {\hat c} \sigma(z_k)+\langle A(z_k),{\hat y}-w \rangle+(c-{\hat c})\sigma(z_k)\\[2.mm]
&\ge& f(x_k,z_k)-\langle A(z_k),{\hat y} \rangle + {\hat c} \sigma(z_k)-\|A(z_k)\|\|{\hat y}-w\|+(c-{\hat c})\sigma(z_k)\\[2.mm]
 &\geq& q({\hat y},{\hat c}) + (c-{\hat c}-\|w-{\hat y}\|)\sigma(z_k),
\end{array}
\]
where we used Cauchy-Schwarz in the second inequality and the definition of $q$ and $(A_0)$ in the third one. Since $\tilde T(s)\subset T$ and $(w,c)\in \tilde T(s)$ we have that $(w,c)\in  T$ and hence $(c-{\hat c}-\|w-{\hat y}\|)>0$. The fact that $q({\hat y},{\hat c})>-\infty$, together with the properness of $\varphi$ imply that $q({\hat y},{\hat c})\in \R$. Altogether, we can re-arrange the last expression to obtain
$$
\sigma(z_k)\leq \frac{s-q({\hat y},{\hat c})}{c-{\hat c}-\|w-{\hat y}\|}=:M(c).
$$
Since $s\ge M_P=M_D\ge q({\hat y},{\hat c})$ we have that $M(c)\ge 0$. We will use now the fact that $(w,c)\in \tilde T(s)$. Indeed, this assumption implies that 
\[
c>\frac{s-q({\hat y},{\hat c})}{K_{\sigma}} + {\hat c}+\| w-{\hat y} \|.
\]
Under this assumption on $(w,c)$ it direct to check that $M(c)<K_{\sigma}$. By Definition \ref{def:assumptions_sigma}(a'), the sequence $\{z_k\}$ is bounded and hence it has a weakly convergent subsequence. Without loss of generality, we can assume that the whole sequence $\{z_k\}$ converges weakly to some $\bar{z}$. Now we proceed to find a subsequence of $\{x_k\}$ which is weakly convergent. Indeed, using the fact that $\{(x_k,z_k)\}\subset lev_{\Phi_{(w,c)}}(s)$ we can write
\begin{equation}\label{levelboundequation}
 f(x_k,z_k)\le s+\langle A(z_k),w \rangle - c\sigma(z_k)
\le s+ \|w\| \|A(z_k)\|
\le s+ \|w\| \sigma(z_k)\le  s+ \|w\| M(c)=: \tilde{\alpha}
\end{equation}
for some $\tilde{\alpha}\in \R$ (note that $(w,c)\in \tilde T(s)$ is fixed). By weak level compactness of $f$ (see Definition
\ref{level-f}), there exists a weakly compact set $B\subset X$ and a weakly open neighbourhood $U$ of $\bar{z}$ such that  
\[
\ds \bigcup_{z\in U} lev_{z,f}(\tilde{\alpha})\subset B,
\]
where $lev_{z,f}(\tilde{\alpha}):=\{x\in X \::\: f(x,z)\le \tilde{\alpha}\}$. Since $\{z_k\}$ converges weakly to $\bar{z}$ and $U$ is weakly open, there exists a $k_0$ such that $z_k\in U$ for all $k>k_0$. Using \eqref{levelboundequation} we deduce that
\[
\{x_k\}_{k>k_0}\subset \bigcup_{k>k_0}\{x\in X\::\:  f(x,z_k)\le \tilde{\alpha}\}\subset B.
\]
Consequently, $\{x_k\}_{k>k_0}\subset B$ and since $B$ is weakly compact, there exists a subsequence of $\{x_k\}_{k>k_0}$ which converges weakly to some $\bar x$. Altogether, we have established that $\{(x_k,z_k)\}$ has a weakly convergent subsequence $\{(x_{k_j},z_{k_j)}\}$, with limit $(\bar x, \bar z)$. Recall that $f$ and $\sigma$ are w-lsc, and the function $\la A(\cdot), w\ra $ is w-usc. Therefore, $-\la A(\cdot), w\ra $ is w-lsc. Altogether, the function $\Phi_{(w,c)}$ given by \eqref{eq:ee1} is w-lsc. For the weakly convergent subsequence $\{(x_{k_j},z_{k_j})\}$ we can write
\[
\Phi_{(w,c)}(\bar x,\bar z)\le \liminf_{j\to \infty} \Phi_{(w,c)}(x_{k_j},z_{k_j})\le s,
\]
where the last inequality follows from the assumption that $\{(x_k,z_k)\}\subset lev_{\Phi_{(w,c)}}(s)$. Hence, we have proved that the weak limit $(\bar x,\bar z)$ belongs to  $lev_{\Phi_{(w,c)}}(s)$, and so the latter set is weakly compact, as claimed. We proceed now to prove the last statement in (iiiB), which requires the existence of $(\tilde x, \tilde z)$ as in \eqref{Lem:eq2} . We note first that, by (ii) and the properness of $\varphi$, $q(w,c)\in \R$ for every $(w,c)\in T$, and hence the same holds for every $(w,c)\in \tilde T(s)$. To establish the equality in \eqref{Lem:eq2}, we need to show that the infimum corresponding to the value $q(w,c)$ is actually attained. We claim that the equality in \eqref{Lem:eq2} follows from the fact that
\begin{equation}\label{claim1}
\argmin_{(x,z)\in X\times H} \Phi_{(w,c)}(x,z)\neq \emptyset.
\end{equation}
Indeed, assume that \eqref{claim1} holds and take  $(\tilde{x},\tilde{z})\in \argmin_{(x,z)\in X\times H} \Phi_{(w,c)}(x,z)$. Thus, 
\begin{equation}\label{mini}
q(w,c)=\inf_{(x,z)\in X\times H} \Phi_{(w,c)}(x,z)=\Phi_{(w,c)}(\tilde{x},\tilde{z})=
f(\tilde{x},\tilde{z}) -\langle A(\tilde{z}),w \rangle + c\sigma(\tilde{z}),
\end{equation}
where we used the definitions of $q$  and  $\Phi_{(w,c)}$, and the assumption on $(\tilde{x},\tilde{z})$.  Therefore, the equality in \eqref{Lem:eq2} will hold if we prove \eqref{claim1}. We know that the set in \eqref{eq:LS} is nonempty, and we proved already that it is weakly compact.  With the notation of Theorem \ref{th:existence}, set $E:=X\times H$, and consider in $E$ the weak topology,  set $\varphi:=\Phi_{(w,c)}$ and $K:=lev_{\Phi_{(w,c)}}(s)$. It holds by definition of level set that $lev_{\Phi_{(w,c)}}(s)\subset \dom \Phi_{(w,c)}$. Altogether, we have that 
\[
\dom \Phi_{(w,c)}\cap lev_{\Phi_{(w,c)}}(s) = lev_{\Phi_{(w,c)}}(s)\neq \emptyset,
\]
where the non-emptiness follows from the first statement in part (iii). Since $\Phi_{(w,c)}$ is w-lsc, all the assumptions of Theorem \ref{th:existence} hold and therefore $\Phi_{(w,c)}$ is bounded below over the set $lev_{\Phi_{(w,c)}}(s)$ and attains its minimum over this set. This establishes \eqref{claim1}, and the proof of the theorem is complete.
\end{proof}


\section{Deflected Subgradient Algorithm (DSG)}\label{sec:DSG}

The following notation will be used throughout the paper.
\[
\begin{array}{l}
  q_{k} :=q(y_{k},c_{k}),\\
  \bar{q}:=q(\bar{y},\bar{c}),
\end{array}
\]
where $(\bar{y},\bar{c})$ represents a dual solution, so $\bar{q}=q(\bar{y},\bar{c})=M_D=M_P\ge q_k$ for every $k$.

\subsection{Definition and Convergence Analysis}

In this section, we define the (DSG) algorithm and establish its convergence properties. We start by defining the Deflected Subgradient Algorithm (DSG).

\begin{algorithm}\label{alg:DSG}
\begin{quote}\rm
  
{\bf Deflected Subgradient Algorithm (DSG)}

{\bf Step $0$}.  Choose $(y_0,c_0)\in H\times \R_+$ such that
$q(y_0,c_0)>-\infty$, and exogenous parameters $\epsilon>0$ (a prescribed tolerance),  $\delta < 1$,
$\{\alpha_k\} \subset (0,\alpha)$ for some $\alpha>0$, and $\{r_k\}\subset\R_+ $ such that $r_k\rightarrow 0$.
Let $k:=0$.

\noindent {\bf Step $1$}. (Subproblem and Stopping Criterion)

$a)$ Find $(x_k,z_k) \in X_{{r}_{k}}(y_k,c_k)$,

$b)$ if $z_k=0$ and $r_k \leq\epsilon$ stop,

$c)$ if $z_k=0$ and $r_k> \epsilon$, then $r_k:=\delta r_k$ and go to $(a)$,

$d)$ if $z_k \neq 0$ go to Step $2$.

\noindent {\bf Step $2$}. (Selection of the stepsize and Updating the Variables)

 Consider $s_k>0$ a stepsize and define

 $y_{k+1}:=y_k - s_kA(z_k)$,

 $c_{k+1}:= c_k + (\alpha_k+1)s_k\sigma(z_k)$,

 $k:=k+1$, go to Step $1$.
  
\end{quote}

\end{algorithm}

\begin{remark}\rm
 Note that, when $A=0$, DSG becomes a classical penalty method. For $A:=I$ the identity map in $H$, we recover IMSg Algorithm defined in \cite[Section 3]{BIMinexact2013}.

\end{remark}

\begin{remark}\rm
\label{Rem:sigma coercive}
By Remark \ref{step1a}, when $r_k>0$, there exists $(x_k,z_k)\in X_{r_k}(y_k,c_k)$ as in Step 1(a), showing that the inexact version of the algorithm is always well defined. When $r_k=0$ for all $k$ we obtain the exact version, which stops at the first $k$ for which $z_k=0$. The well-definedness of the exact version is shown below in Proposition \ref{exactDSG}. In the latter result, we give conditions under which Step 1(a) of the exact version can be performed. Our analysis includes a choice of $\sigma$ either as in part (a) or as in part (a'), of Definition \ref{def:assumptions_sigma}.  

\begin{remark}\label{rem:eps-sub}\rm
 Using Proposition \ref{gradz0}, we see that Step 2 in Algorithm \ref{alg:DSG} is nothing but an epsilon subgradient step for the maximization of the dual function.  
\end{remark}

\begin{remark}\rm
If $z_k\in N(A)$, then $y_{k+1}=y_k$ and $c_{k+1}>c_k$. Therefore $y_k$ is not updated in this case. The situation in which $z_k\in N(A)$ does not  pose a problem in terms of convergence. Indeed, our results hold for $A=0$, so that $N(A)=X$.
\end{remark}

%
\end{remark}


We establish next properties that hold for every stepsize $s_k>0$ and for every $\alpha_k\in (0,\alpha)$. The proof of the equivalence between statements (a) and (b) is standard, and, with minimal changes, follows the same steps of \cite[Proposition 3.1(ii)]{BIMinexact2013}. We include its short proof here, however, because the expressions involved in the proof will often be used in later results.

\begin{proposition}[Characterization of dual convergence]\label{dualSeq}
\noindent Let $\{(x_k,z_k)\}$ and $\{(y_k,c_k)\}$ be the sequences generated by {\rm DSG}, and assume that ${\bf (A_0)}$ holds. The following statements are equivalent.
    \begin{itemize}
        \item[(a)]  The dual sequence $\{(y_k,c_k)\}$ is bounded.
        \item[(b)] $\sum_{k}
s_k\sigma(z_k) < +\infty.$
        \item[(c)] The dual sequence $\{(y_k,c_k)\}$ converges strongly to a limit.
         \item[(d)]  The sequence $\{c_k\}$ is Cauchy.
         \item[(e)] The sequence $\{c_k\}$ is bounded.
    \end{itemize}
    Furthermore, if  $\{c_k\}$ is bounded, then $\{y_k\}$ is also bounded.
\end{proposition}
\begin{proof}
Using ${\bf (A_0)}$ and the definition of $\{y_k\}$, we obtain
\begin{equation}\label{eqdualseq1}
\|y_{k+1}-y_0\|\le \sum_{j=0}^k \|y_{j+1}-y_j\|= \sum_{j=0}^k s_j \|A(z_j)\|\le \sum_{j=0}^k s_j \sigma(z_j).
\end{equation}
On the other hand, by definition of $\{c_k\}$ we have
\begin{equation}
\label{eqdualseq2} c_{k+1} - c_0 = \sum_{j=0}^kc_{j+1}-c_j=\sum_{j=0}^k
(\alpha_j + 1)s_j\sigma(z_j)\le (\alpha + 1)\sum_{j=0}^k
s_j\sigma(z_j),
\end{equation}
where we used the fact that $\alpha_k<\alpha$ for every $k$. We prove first the equivalence between (a) and (b). If (b) holds, then \eqref{eqdualseq1} and \eqref{eqdualseq2} readily yield (a). Conversely, assume that (a) holds.  Using the left hand side of  \eqref{eqdualseq2} and (a) gives the existence of $M>0$ such that
\begin{equation}
\label{eqdualseq3} 
M\ge c_{k+1} - c_0 = \sum_{j=0}^kc_{j+1}-c_j=\sum_{j=0}^k
(\alpha_j + 1)s_j\sigma(z_j)\ge \sum_{j=0}^k s_j\sigma(z_j),
\end{equation}
where we used the fact that $\alpha_j>0$ for all $j$. Since the inequality above holds for all $k$, we must have $\sum_{k}^{\infty}
s_k\sigma(z_k) \le M$ and we deduce (b).
Let us now show that (a) is equivalent to (c). Clearly (c) implies (a), so it is enough to show that (a) implies (c). Assume that (a) holds. Then the sequence $\{c_k\}$ is bounded. Since it is strictly increasing, it must be convergent. In particular, this implies that the sequence $\{c_k\}$ is Cauchy. We will show now that $\{y_k\}$ is also Cauchy with respect to the norm. Indeed, for every $k,j\in \dN$ we can write
\begin{equation}
\label{eqdualseq4} 
c_{k+j} - c_k = \sum_{l=k}^{k+j-1} c_{l+1}-c_l=\sum_{l=k}^{k+j-1}
(\alpha_l + 1)s_l\sigma(z_l)\ge \sum_{l=k}^{k+j-1}
s_l\sigma(z_l) \ge \| y_{k+j} - y_k\|,
\end{equation}
where we used again the fact that $\alpha_j>0$ for all $j$ in the first inequality. The last inequality is obtained as in \eqref{eqdualseq1}, but with $k+j-1$ in place of $k$ and $k$ in place of $0$. Since $\{c_k\}$ is Cauchy, then for all $j\in \dN$ we have
\begin{equation}
\label{eqdualseq5} 
0= \lim_{k\to\infty} c_{k+j} - c_k\ge \lim_{k\to\infty}\| y_{k+j} - y_k\|\ge 0,
\end{equation}
so $\{y_k\}$ is also Cauchy as claimed. Our claim is true and since $H$ is complete, the sequence $\{y_k\}$ strongly converges to a limit $\bar y$. Altogether, $\{(y_k,c_k)\}$ is strongly convergent, so (a) implies (c). Since (c) implies (d), to complete the proof it is enough to show that (d) implies (c). This is achieved in a similar way as in (a) implies (c). Indeed, if $\{c_k\}$ is Cauchy, then by \eqref{eqdualseq5} we deduce that $\{y_k\}$ is also Cauchy, and hence by completeness of $H\times \R$, we deduce that (c) holds. We clearly have that (d) implies (e). If (e) holds, then by \eqref{eqdualseq2} and \eqref{eqdualseq1} we must have $\{y_k\}$ also bounded, hence (a) holds.  The proof is complete.
\medskip
\end{proof}
\medskip

The next result establishes the well-definedness of the algorithm, namely that the minimization performed in Step 1 has a solution. We establish this fact either when $\sigma$ is coercive or when it is as in Definition \ref{def:assumptions_sigma}(a'). The proof of part (i) in the next result follows the steps of \cite[Proposition 3.2]{BIMinexact2013}, so we omit its proof. Part  (ii) uses the weaker assumption on $\sigma$, namely conditional coercivity.

\begin{proposition}[Well-definedness of DSG]\label{exactDSG}
Consider $s\ge M_P$, $({\hat y},{\hat c})$, and the sets $T,\tilde T(s)$ as in Theorem \ref{th:compact_level_set_Lagrangian}. 
\begin{itemize}
\item[(i)] Assume that $\sigma$ verifies Definition \ref{def:assumptions_sigma}(a). Take $y_0:={\hat y}$ and $c_0>{\hat c}$. Then, the dual  sequence $\{(y_k,c_k)\}$ generated by the exact version of DSG with $(y_0,c_0)$  is well-defined. Namely, the set 
\[
X(y_k,c_k)\\:=\{(x,z) \in X\times H:
f(x,z)-\langle A(z),y_k\rangle + c_k\sigma(z)
= q(y_k,c_k)\},
\]
is nonempty for all $k\ge0$. 
\item[(ii)] Assume that $\sigma$ verifies Definition \ref{def:assumptions_sigma}(a'). Take $y_0:={\hat y}$ and $c_0>{\hat c} +\frac{s-q({\hat y}, {\hat c})}{K_{\sigma}}$. Then the same conclusion in (i) holds.
\end{itemize}
In particular, as long as $(y_0,c_0)$ is chosen as in (i) or (ii) for the corresponding type of $\sigma$, we will have that  the exact version of DSG is well defined; that is to say, for all $k\ge1$, there exists $(x_k,z_k)\in X\times H$ satisfying $ q(y_k,c_k)=[f(x_k,z_k)-\langle A(z_k),y_k\rangle + c_k\sigma(z_k)]\in \R$ for every $k\ge 1$.
\end{proposition}
\begin{proof} (i) Similar to \cite[Proposition 3.2]{BIMinexact2013}.
(ii) By assumption $(y_0,c_0)\in \tilde T(s)$, we know by Theorem \ref{th:compact_level_set_Lagrangian}(iiiB) that there exists $(x_0,z_0)\in X(y_0,c_0)$. 
If $z_0=0$ the algorithm stops at $k=0$ and the claim in (ii) holds for the single iterate $(y_0,c_0)$. Assume that $k\ge 1$ (i.e., the algorithm does not stop at $k=0$ and hence $z_0\not=0$).  Let us show that $(y_k,c_k)\in \tilde T(s)$ for every $k\ge 1$. From \eqref{eqdualseq1}, we have $\sum_{j=0}^{k-1} s_j \sigma(z_j)\ge \|y_{k}-y_0\|$  and from \eqref{eqdualseq2} we have $ c_{k} - c_0 =\sum_{j=0}^{k-1}
(\alpha_j + 1)s_j\sigma(z_j)$. Altogether, we have
\[
\begin{array}{rcl}
c_k&=&c_0+\sum_{j=0}^{k-1}(\alpha_j + 1)s_j\sigma(z_j)
\ge c_0+ \|y_{k}-y_0\|+\sum_{j=0}^{k-1}\alpha_j s_j\sigma(z_j)\\
&&\\
&&>c_0+ \|y_{k}-y_0\|>({\hat c} +\dfrac{s-q({\hat y}, {\hat c})}{K_{\sigma}})+ \|y_{k}-y_0\|,
\end{array}
\]
where the first strict inequality follows from the fact that $z_j\not=0$ (equivalently, $\sigma(z_j)\not=0$) for every $0\le j\le k-1$ (otherwise the algorithm would have stopped at some $j< k$), and the second strict inequality uses the definition of $c_0$. Therefore, $(y_k,c_k)\in  \tilde T(s)$ for every $k\ge 1$ and the result follows from Theorem \ref{th:compact_level_set_Lagrangian}(iiiB).The last assertion of the proposition follows directly from (i) and (ii).
\end{proof}

\bigskip
Part (a) of the following result is proved for $A=I$ and $\sigma$ coercive in \cite[Lemma 3.2]{BIMinexact2013}, and establishes the boundedness of the sequences $\{z_k\}$  and $\{\sigma(z_k)\}$ without any additional assumptions on the parameters of DSG.  Since the proof for the case involving the map $A$ and $\sigma$  as in  Definition \ref{def:assumptions_sigma}(a) follows the same steps as the ones in \cite[Lemma 3.2]{BIMinexact2013}, we omit its proof, and analyze below the case for $\sigma$ conditionally coercive.

\begin{proposition}\label{sigma_k bounded}
Assume that ${\bf (A_0)}$ holds and that $z_0\not=0$. Fix $\tilde r$ an upper bound of $\{r_k\}$. 
\begin{itemize}
\item[(a)] If $\sigma$ is as in Definition \ref{def:assumptions_sigma}(a) and $(y_0,c_0)$ in DSG Algorithm is taken as in Proposition \ref{exactDSG}(i), then  the sequences $\{z_k\}$  and $\{\sigma(z_k)\}$ are bounded.
\item[(b)] If $\sigma$ is conditionally coercive with constant $K_{\sigma}$ and $(y_0,c_0)$ in DSG Algorithm is taken as in Proposition \ref{exactDSG}(ii). Then, the sequence $\{\sigma(z_k)\}$ is bounded. Furthermore,  if the parameters $\alpha_0,\, s_0$ in DSG are chosen such that 
\begin{equation}\label{eq:K}
\alpha_0 s_0 >\frac{M_P-q(y_0,c_0)+\tilde r}{K_{\sigma}\,\sigma(z_0)},
\end{equation}
then $\{z_k\}$ is bounded.
\end{itemize}
\end{proposition}
\begin{proof} (a) Similar to \cite[Lemma 3.2]{BIMinexact2013}. Let us prove the first statement in (b), namely the boundedness of $\{\sigma(z_k)\}$. If the algorithm stops at iteration $k_0$, then the sequences $\{\sigma(z_k)\}$ and $\{z_k\}$ are finite and therefore bounded.  Indeed, in the latter case, the sequence either stops (if $r_{k_0}\le \epsilon $), or it goes into a finite inner loop until $r_{k_0}\le \epsilon $. In either case, the sequences $\{z_k\}$  and $\{\sigma(z_k)\}$ are finite and their boundedness trivially holds. Hence, it is enough to assume that Step 2 is visited at every $k\ge 0$ and hence $z_k\neq 0$ for every $k\ge 0$.  Call $a_0:= \alpha_0 s_0\sigma(z_0)>0$. From \eqref{eqdualseq1} we deduce for all $k\geq 1$,
\[
\begin{array}{rcl}
c_{k}-c_0&=&\sum_{l=0}^{k-1}
(\alpha_l + 1)s_l\sigma(z_l)=

\sum_{l=0}^{k-1} s_l\sigma(z_l) + \sum_{l=0}^{k-1} \alpha_l s_l \sigma(z_l)\\
&&\\
&&\ge \sum_{l=0}^{k-1} s_l\sigma(z_l) + \alpha_0 s_0\sigma(z_0)\ge \sum_{l=0}^{k-1} s_l\| Az_l \| + \alpha_0 s_0\sigma(z_0)\\
&&\\
&&\ge \|y_k - y_0 \| + \alpha_0 s_0\sigma(z_0)=\|y_k-y_0\|
+a_0,
\end{array}
\]
where we used ${\bf (A_0)}$ in the second inequality, \eqref{eqdualseq1} in the last one, and the definition of $a_0$ in the last equality. Re-arrange this expression to obtain
\begin{equation}\label{ee7}
c_k - c_0 - \|y_k-y_0\| \geq a_0.
\end{equation}
By Proposition \ref{gradz0}(i), we know that
$(-A(z_k),\sigma(z_k)) \in \partial_{r_k} q(y_k,c_k)$. Use the subgradient inequality to write, for every $k$,
\[
\begin{array}{rcl}
-\infty< q_0=q(y_0,c_0)&\leq& q(y_k,c_k) + \langle -A(z_k),y_0-y_k\rangle +(c_0-c_k)\sigma(z_k) + r_k
\\[3mm] &\leq& q_k + \|A(z_k)\|\|y_k-y_0\| + (c_0-c_k)\sigma(z_k) +r_k
\\[3mm] &\leq& q_k  + \sigma(z_k)\left(\|y_k-y_0\| + c_0 - c_k\right) + \tilde{r}
\\[3mm] &\leq& q_k  - a_0\sigma(z_k) + \tilde{r}\le q_k + \tilde{r},
\end{array}
\]
where we used Cauchy-Schwarz inequality, ${\bf (A_0)}$, and \eqref{ee7}. The above expression yields the boundedness of $\{\sigma(z_k)\}$. Indeed, it re-arranges to
\[\sigma(z_k)\leq \frac{M_D-q_0 +\tilde{r}}{a_0}:=b.\]
Hence $\sigma(z_k)\leq b$ for all $k$ and the proof of the first statement is complete. Let us prove now that, if \eqref{eq:K} holds, then we also have that 
$\{z_k\}$ is bounded.  Indeed,  \eqref{eq:K} directly implies that $a_0=\alpha_0 s_0\sigma(z_0) >\frac{M_D-q_0 +\tilde{r}}{K_{\sigma}}$ and hence the above expression becomes
\[
\sigma(z_k)\leq \frac{M_D-q_0 +\tilde{r}}{a_0} <K_{\sigma},
\]
which implies that $\{z_k\}$ is bounded by definition of $K_{\sigma}$. The proof is complete.
\end{proof}

\medskip

We show next that, if an iterate generated by DSG is a dual solution, then the exact version of DSG must stop, either at the current iteration or at the next one. This result holds for either type of $\sigma$. 

\begin{proposition}\label{prop:dualsol}
Assume that ${\bf (A_0)}$ holds and assume DSG has $r_k=0$ for all $k$. If the $k$th DSG iterate is a dual solution, then either $z_k=0$ or $z_{k+1}=0$. Consequently, in this situation DSG will stop at iteration $k$ or $k+1$.
\end{proposition}
\begin{proof}
Assume that, at iteration $k$, we have that $(y_k,c_k)\in S(D)$. This means that $q_k=q(y_k,c_k)=M_D$. It is enough to prove that, if $z_k\neq 0$, then $z_{k+1}=0$. Assume that $z_k\neq 0$, by ${\bf (A_0)}$, we clearly have that
\begin{equation}\label{prop:eq1}
\sigma({z}_k) \sigma({z}_{k+1}) -\|A{z}_k \|\,\|A {z}_{k+1} \|\ge 0.
\end{equation}
Take $({x}_{k+1}, {z}_{k+1})\in X(y_{k+1},c_{k+1})$. 
With the notation of DSG, denote $\varepsilon_k:=\alpha_k s_k$. 
Using the fact that $(y_k,c_k)\in S(D)$ and the definitions of $q$ and $({x}_{k+1}, {z}_{k+1})$, we can write
\[
\begin{array}{rcl}
 M_D\ge q_{k+1}&=&f({x}_{k+1}, {z}_{k+1})-\langle A {z}_{k+1},y_{k+1}  \rangle + c_{k+1} \sigma({z}_{k+1})  \\
 &&\\
  &=&  f({x}_{k+1}, {z}_{k+1})-\langle A {z}_{k+1},y_{k}-s_k Az_k  \rangle + (c_{k} +(s_k+\varepsilon_k)\sigma({z}_k)) \sigma({z}_{k+1})\\
 &&\\
  &=&  f({x}_{k+1}, {z}_{k+1})-\langle A {z}_{k+1},y_{k}\rangle +s_k \langle A {z}_{k+1}, Az_k  \rangle + (c_{k} +(s_k+\varepsilon_k)\sigma({z}_k)) \sigma({z}_{k+1})\\

  &&\\
   &\ge &  \left[ f({x}_{k+1}, {z}_{k+1})-\langle A {z}_{k+1},y_{k}  \rangle + c_{k} \sigma({z}_{k+1})\right]  + s_k \left(\sigma({z}_k) \sigma({z}_{k+1}) - \| A{z}_k\|\, \|A {z}_{k+1}\| \right)\\
    &&\\
   &&   +\varepsilon_k \sigma({z}_k)\sigma({z}_{k+1})\ge  q_k +\varepsilon_k \sigma({z}_k)\sigma({z}_{k+1})=M_D +\varepsilon_k \sigma({z}_k)\sigma({z}_{k+1}),
\end{array}
\]
where we used the definition of DSG in the third equality. We also used \eqref{prop:eq1} and the definition of $q_k$ in the second to last inequality. This shows that $\varepsilon_k \sigma({z}_k)\sigma({z}_{k+1})\le 0$. Since both $\varepsilon$ and $\sigma({z}_k)$ are assumed to be positive, we must have $\sigma({z}_{k+1})=0$ and hence $z_{k+1}=0$.
\end{proof}
\medskip

The following theorem states that DSG guarantees a monotonic increase of the dual function. If the initial iterate is taken as in Theorem \ref{th:compact_level_set_Lagrangian}, we know that the algorithm is well defined for either type of $\sigma$ (coercive or conditionally coercive). Assuming this is the case, the proof of the result below follows similar steps to those in \cite[Theorem 3.1]{BIMinexact2013} and hence are omitted. 

\begin{theorem}\label{increase} 
Assume that {\rm DSG} generates an infinite sequence $\{(y_k,c_k)\}$ and that for every $k$, $(y_k,c_k)$ is not a dual solution.  Then $q(y_{k+1},c_{k+1})>q(y_k,c_k)$.
\end{theorem}
\begin{proof} Similar to \cite[Theorem 3.1]{BIMinexact2013}. 
\end{proof}

From now on, we assume that $z_k\neq0$ for all $k$. In other words, we assume that the method generates an infinite sequence. We will also assume that the initial  iterate and parameters are chosen so that, for either type of $\sigma$, the previous results and properties hold. The technical result below has a proof similar to the one in \cite[Lemma 3.3]{BIMinexact2013} and hence is omitted.

\begin{lemma}\label{estimate}
Consider the sequences
$\{(x_k,z_k)\}$, $\{(y_k,c_k)\}$ generated by {\rm DSG} algorithm.
\begin{itemize}
\item [{\rm(a)}] The following estimates hold for all $k\geq 1$
\begin{eqnarray}
\label{lemequat1} f(x_k,z_k)-\langle A(z_k),y_0\rangle&\leq& q_k +r_k,\;\mbox{and}\\[3mm]
\label{lemequat2} \sigma(z_k)\sum_{j=0}^{k-1}\alpha_js_j\sigma(z_j)&\leq& q_k-q_0+r_k.
\end{eqnarray}
\item [{\rm(b)}] Assume that the dual solution set $S(D)$ is nonempty. If $(\bar{y},\bar{c})\in S(D)$
then for all $k$,
\begin{equation}\label{lemequat3}
\|y_{k+1}-\bar{y}\|^2 \leq  \|y_k-\bar{y}\|^2 + 2s_k\sigma(z_k)\left[\frac{s_k\sigma(z_k)}{2} + \ds\frac{q_k-\bar{q}+r_k}{\sigma(z_k)} + \bar{c}-c_k \right].
\end{equation}
\end{itemize}
\end{lemma}
\begin{proof} Similar to  \cite[Lemma 3.3]{BIMinexact2013}.
\end{proof}

The following result holds for either type of $\sigma$. The only new result involved in its proof is the fact that, for $\sigma$ conditionally coercive, strong duality holds. Again, due to the similarity of the proof techniques with  \cite[Lemma 3.4]{BIMinexact2013}, we omit its proof here.

\begin{lemma}\label{zPrimalCvg}If the sequence
$\{z_k\}$ converges weakly to $0$, then  $\{q_k\}$ converges to $\bar{q}$,
the primal sequence $\{x_k\}$ is bounded, and all its weak accumulation points are primal solutions.
\end{lemma}
\begin{proof} Similar to  \cite[Lemma 3.4]{BIMinexact2013}.
\end{proof}


\subsection{Algorithm DSG-1}\label{ssec:alg1}
We consider in this section the stepsize similar to the one given in \cite[Algorithm 1]{BIMinexact2013}, and use it for our particular scheme. The difference is the use of the function $A$ in the choice of the stepsize (see the definition of $\eta_k$ below). Take two parameters $\beta>\eta>0$.  We consider the step size
\begin{equation}\label{sk_DSG1}
s_k\in [\eta_k,\beta_k],
\end{equation}
where $\eta_k:=\min\{\eta,\|A(z_k)\|+\|z_k\|\}$ and $\beta_k:=\max\{\beta,\sigma(z_k)+\|z_k\|\}$. 
With this choice of $s_k$, we denote the DSG~algorithm as DSG-1.
\begin{remark}\rm
If ${\bf (A_0)}$ holds, then
\[\eta_k\leq \|A(z_k)\|+\|z_k\|
\leq \sigma(z_k)+\|z_k\|\leq\beta_k,
\]
where first and last inequalities use the definition of $\eta_k,\beta_k$. The second inequality holds by ${\bf (A_0)}$. Note that, a constant stepsize for all iterations is admissible.
\end{remark}

\noindent The next theorem only requires a $\sigma$ which satisfies the following property:
\begin{equation}\label{prop}
\hbox{ if } \sigma(w_k)\downarrow 0 \hbox{ then }\{w_k\} \hbox{ bounded.}
\end{equation} 
Its proof considers two possible cases, according to whether the dual sequence $(y_k,c_k)$ is bounded or not. The case of an unbounded sequence has a proof similar to the one \cite[Theorem 3.2]{BIMinexact2013}, and hence is omitted. The case of bounded dual sequence $(y_k,c_k)$ is slightly different because of our different type of stepsize, so we provide it here.

\begin{theorem}\label{primalCvg1}
Assume that $\sigma$ is an augmenting function verifies that  if $\sigma(w_k)\downarrow 0$, then $\{w_k\}$ bounded, and assume that $M_P=M_D$. Consider the primal sequence $\{x_k\}$ generated by {\rm DSG-1}. Take the parameter sequence $\{\alpha_k\}$ satisfying $\alpha_k\ge\bar \alpha$ for all $k$ and some $\bar \alpha>0$. Then $\{x_k\}$ is bounded, all its weak accumulation points are primal solutions, and $\{q_k\}$ converges to the optimal value $M_P$.
\end{theorem}

\begin{proof}
Take the dual sequence $\{(y_k,c_k)\}$ generated by DSG-1. If $\{(y_k,c_k)\}$ is unbounded, then the proof is similar to the corresponding part of \cite[Theorem 3.2]{BIMinexact2013}. We proceed to consider the case in which $\{(y_k,c_k)\}$ is bounded. By Proposition \ref{dualSeq} $(i)$, $\sum_k s_k\sigma(z_k)<\infty$. In particular, $\{s_k\sigma(z_k)\}$ converges to $0$. On the other hand, from $s_k\ge\min\{\eta,\|A(z_k)\|+\|z_k\|\}$, we obtain
\[
s_k\sigma(z_k)\geq \min\{\eta\sigma(z_k),(\|A(z_k)\|+\|z_k\|)\sigma(z_k)\}
\geq \min\{\eta\sigma(z_k),\|z_k\|\sigma(z_k)\}> 0,
\]
because $z_k\not=0$ for all $k$. Since $\eta>0$, we conclude that $\{\|z_k\|\}$ converges to $0$ or $\sigma(z_k)$ converges to $0$. We will show that either case implies that $\{z_k\}$ weakly converges to $0$. If $\{\|z_k\|\}$ converges to $0$ then $\{z_k\}$ converges strongly to $0$, and hence weakly to $0$. Alternatively, if $\sigma(z_k)$ converges to $0$, then $\{z_k\}$ is bounded by assumption. Then, there exists a subsequence $\{z_{k_j}\}$ weakly converging to some $\tilde{z}$. From the weak lower semicontinuity of $\sigma(\cdot)$, we have $0\le\sigma(\tilde{z})\le\liminf_{k\to\infty}\sigma(z_{k_j})=0$. Hence $\sigma(\tilde{z})=0$. The properties of $\sigma$ now imply that $\tilde{z}=0$. Therefore, the whole sequence $\{z_k\}$ weakly converges to $0$. Thus, in the case that $\{(y_k,c_k)\}$ is bounded, the results follows from Lemma \ref{zPrimalCvg} and the zero duality gap property $\bar q=M_P$. 
%
\end{proof}

The following corollary holds because a conditionally coercive $\sigma$ induces strong duality and also verifies  \eqref{prop}.

\begin{corollary}
If $\sigma$ verifies Definition \ref{def:assumptions_sigma}(b) with either (a) or (a'),  then  the conclusion of Theorem \ref{primalCvg1} holds.
\end{corollary}
\begin{proof}
Since  a conditionally coercive $\sigma$ verifies  \eqref{prop}, the same holds for a coercive $\sigma$. By Theorem \ref{th:StrongDual}, strong duality holds. So we are in conditions of Theorem \ref{primalCvg1}.
\end{proof}

Theorem \ref{primalCvg1} above establishes primal convergence results for DSG-1, the following theorem establishes a dual convergence result, its proof is identical to \cite[Theorem 3.3]{BIMinexact2013} and hence omitted.

\begin{theorem}\label{dualCvg1}
If DSG-1 generates an infinite sequence $\{(y_k,c_k)\}$, then every weak accumulation point of $\{(y_k,c_k)\}$, if any, is a dual solution.
\end{theorem}
\begin{proof} See \cite[Theorem 3.3]{BIMinexact2013}.
\end{proof}

We know that, when $\{c_k\}$ is bounded, then $\{y_k\}$ bounded by Proposition \ref{dualSeq}. The converse is not necessarily true, and it holds under an additional assumption which requires the sequences $\{r_k\}$ and $\sigma(z_k)$ to decrease at a similar rate.
\medskip

${\bf(\Gamma_0)}$: There exists $R>0$ such that
$r_k\leq R\sigma(z_k)$ for all $k$, that is $r_k\approx O(\sigma(z_k))$.

The proof of the next result is similar to \cite[Lemma 3.5]{BIMinexact2013} and hence omitted. 

\begin{lemma}\label{ycBounded} Assume that ${\bf(A_0)}$ and ${\bf(\Gamma_0)}$ hold. If the dual solution set is nonempty and $\{y_k\}$ bounded, then $\{c_k\}$ is bounded too.
\end{lemma}
\begin{proof} Similar to \cite[Lemma 3.5]{BIMinexact2013}.
%
%
\end{proof}

\begin{remark}\rm
Lemma \ref{ycBounded} holds under assumptions ${\bf(A_0)}$ and ${\bf(\Gamma_0)}$ in the general framework of DSG, regardless the choice of the stepsize $s_k$.
\end{remark}

The following result extends \cite[Proposition 3.3]{BIMinexact2013} to our general case, since our function $\sigma$ and our stepsize are different, it requires a slightly different proof.

\begin{proposition}\label{seq-dual}
Assume that ${\bf(\Gamma_0)}$ holds, and that we are in conditions of Proposition \ref{sigma_k bounded}(a) or (b). Assume also that {\rm DSG-1} generates an infinite dual sequence $\{(y_k,c_k)\}$. If the dual optimal set is nonempty then $\{(y_k,c_k)\}$ is bounded.
\end{proposition}

\begin{proof} Under the conditions of Proposition \ref{sigma_k bounded}(a) or (b), we have that  $\{\sigma(z_k)\}$ and $\{z_k\}$ are bounded, so take $b>0$ such that $\sigma(z_k)+\|z_k\|<b$ for all $k$. By definition, $s_k\le\beta_k\le\max\{\beta,b\}=:\hat b$. In particular, $s_k\sigma(z_k)\le\ b \hat{b}=\bar{b}$ for all $k$. Let $R$ be as in ${\bf(\Gamma_0)}$ and take $(\bar y,\bar c)\in S(D)$. We claim that $\{(y_k,c_k)\}$ is bounded. We will show firstly that $\{c_k\}$ is bounded. Suppose by contradiction that $\{c_k\}$ is unbounded. Thus there exists $k_0$ such that $c_k\ge M:=\frac{\bar b}{2}+R+\bar c$ for all $k\ge k_0$. Observing that $q_k\le\bar q$ and using the estimates in \eqref{lemequat3}, we obtain
\begin{equation}
\begin{array}{rcl}
\|y_{k+1}-\bar{y}\|^2
&\leq&  \|y_k-\bar{y}\|^2 + 2s_k\sigma(z_k)\left[\ds\frac{s_k\sigma(z_k)}{2} + \ds\frac{q_k-\bar{q}+r_k}{\sigma(z_k)} + \bar{c}-c_k \right]\\
\\
&\le& \|y_k-\bar{y}\|^2 + 2s_k\sigma(z_k)\left[\ds\frac{\bar b}{2} + \ds\frac{r_k}{\sigma(z_k)} + \bar{c}-c_k \right]\\
&&\\
&\le& \|y_k-\bar{y}\|^2 + 2s_k\sigma(z_k)\left[\ds\frac{\bar b}{2} + R + \bar{c}-c_k \right]\\
&&\\
&\le& \|y_{k}-\bar{y}\|^2
\end{array}
\end{equation}
for all $k\ge k_0$. It follows that $\{\|y_k-\bar y\|\}$ is a decreasing sequence and hence $\{y_k\}$ is bounded. By Lemma \ref{ycBounded}, this entails a contradiction. Therefore, the dual sequence is bounded.
\end{proof}

Theorem \ref{dualconv}, which we prove next, establishes strong convergence of the whole dual sequence generated by DSG-1
to a dual solution. Theorem 3.4 in \cite{BIMinexact2013} relies on Fej\'er convergence properties and establishes only weak convergence of the dual sequence. Our proof is inspired by Theorem 5.1 in \cite{BIM2015} and uses the properties of $q$.

{
\begin{theorem}\label{dualconv}
If ${\bf (A_0)}$ holds and the parameter sequence $\{\alpha_k\}$ satisfies that $\alpha_k\ge\bar\alpha$ for all $k$ and some $\bar\alpha>0$, then the following hold.\hfill
\begin{itemize}
    \item[(i)] Assume that $\sigma$ is an augmenting function that verifies \eqref{prop}, and assume that $M_P=M_D$.  If the dual sequence generated by {\rm DSG-1} is bounded then $S(D)\neq {\varnothing}$ and the dual sequence converges strongly to a dual solution. 
    \item [(ii)] Assume that ${\bf (\Gamma_0)}$ holds,  and that we are in conditions of Proposition \ref{sigma_k bounded}(a) or (b). If $S(D)\neq {\varnothing}$, then the dual sequence generated by {\rm DSG-1} is strongly convergent to some dual solution.
\end{itemize}
\end{theorem}

\begin{proof}
(i) Since the dual sequence $\{(y_k,c_k)\}$ is bounded, it converges strongly to some $(\tilde y,\tilde c)$ by Proposition \ref{dualSeq}(a)(c). We only need to prove now that the limit $(\tilde y,\tilde c)\in S(D)$. Indeed, we can write
 \begin{equation}\label{eq:dualCvg1}
 M_P=\lim_k q(y_k,c_k)\le q(\tilde y,\tilde c)\le M_P,
\end{equation}
where the equality follows from Theorem \ref{primalCvg1} and strong duality (see Theorem \ref{th:StrongDual}), note that we can write a limit because the sequence $\{q_k\}$ is increasing. The first inequality follows from the fact that $q$ is weakly (and hence strongly) upper semicontinuous, and the last inequality is a consequence of the fact that $q(\tilde y,\tilde c)\le M_D\le M_P$. Therefore, we showed that  $(\tilde y,\tilde c)\in S(D)$ and the proof of (i) is complete.
We now show (ii). Since the dual set is non-empty, by Proposition \ref{seq-dual}, the dual sequence $\{(y_k,c_k)\}$ is bounded. By (i), we have $\{(y_k,c_k)\}$ converges strongly to some $(\tilde y,\tilde c)\in S(D)$.
\end{proof}

\medskip

The following straightforward corollary characterizes the existence of dual solutions.

\begin{corollary}\label{boundedDualIFF}
If ${\bf (A_0)}$ and ${\bf (\Gamma_0)}$ hold, and that we are in conditions of Proposition \ref{sigma_k bounded}(a) or (b). Assume that $\alpha_k\ge\bar\alpha>0$ for all $k$. The following statements are equivalent.\hfill
\begin{itemize}
    \item [(a)] The dual sequence generated by {\rm DSG-1} is bounded.
    \item [(b)] The dual set is not empty.
\end{itemize}
\end{corollary}
\begin{proof}
The proof follows directly from the fact that the assumptions ensure that we are in conditions of both parts (i) and (ii) in Theorem \ref{dualconv}. 
\end{proof}
}
\subsection{Algorithm DSG-2}\label{sec:alg2}
In this section we adopt the same stepsize proposed in \cite[Algorithm 2]{BIMinexact2013}, which ensures that DSG converges in a finite number of steps. We show in this section that these convergence results are preserved when using the map $A$ in the Lagrangian.

\noindent Take  $\beta>0$ and a sequence $\{\theta_k\} \subset\R_+$  such that
$\sum_j\theta_j=\infty$, and $\theta_k\leq\beta$ for all $k$.
Consider the step size
\begin{equation}\label{sk_DSG2}
s_k\in [\eta_k,\beta_k],
\end{equation}
where $\eta_k:= \theta_k/\sigma(z_k)$ and $\beta_k:= \beta/\sigma(z_k)$. 
DSG with this stepsize selection is denoted by DSG-2.

The following result extends  \cite[Theorem 3.5]{BIMinexact2013}. By looking carefully at the proof in \cite[Theorem 3.5]{BIMinexact2013}, it is seen that the assumptions can be weakened by just assuming that $\sigma$ verifies \eqref{prop}. Since the proof is similar to that in \cite[Theorem 3.5]{BIMinexact2013}, we omit it.

\begin{theorem}\label{primalCvg}
Let $\{(x_k,z_k)\}$ and $\{(y_k,c_k)\}$ be the sequences generated by {\rm DSG-2}. Suppose that the parameter sequence $\alpha_k\ge \bar \alpha>0$. Assume that $\sigma$ verifies \eqref{prop}. Then only one of the following cases occurs:\\
\noindent (a) There exists a $\bar{k}$
such that {\rm DSG-2} stops at iteration $\bar{k}$. As a
consequence $x_{\bar{k}}$ and $(y_{\bar{k}},c_{\bar{k}})$ are
$\epsilon$-optimal primal and $\epsilon$-optimal dual
solutions, respectively. In this situation $\{(y_k,c_k)\}$ must be bounded.

\noindent (b) The dual sequence $\{(y_k,c_k)\}$ is unbounded. In this case, $\{z_k\}$ converges weakly to $0$, $\{q_k\}$ converges to $M_P$, the primal sequence $\{x_k\}$ is bounded and all its weak accumulation points are primal solutions.
\end{theorem}
\begin{proof} Similar to \cite[Theorem 3.5]{BIMinexact2013}.
\end{proof}
This directly gives the following result.

\begin{corollary}\label{boundedDualIFF-2}
Let $\{(x_k,z_k)\}$ and $\{(y_k,c_k)\}$ be the sequences generated by {\rm DSG-2}. Suppose that the parameter sequence $\alpha_k\ge \bar \alpha>0$. Assume that $\sigma$ is either coercive or conditionally coercive. Then the conclusion of Theorem \ref{primalCvg} hold.
\end{corollary}
\begin{proof}
The claim follows from the fact that both coerciveness and conditional coerciveness imply condition \eqref{prop}.
\end{proof}

\section{Acknowledgements}
We thank C. Yalcin Kaya for useful discussions on various aspects of the DSG algorithm.

\section{Concluding Remarks}\label{sec:conclusion}
We provide a detailed analysis of primal-dual problems in infinite dimensions, when the dual problem is obtained by using a Lagrangian that involves a map $A$ in the linear term (see \eqref{eq:Lagrn} in Definition \ref{def:basic}), and an augmenting function which, in some cases, does not need to be coercive (see Definition \ref{def:assumptions_sigma}(a')). For such primal-dual pairs, we establish strong duality. Moreover, we show that the inexact DSG method has the same primal and dual convergence properties as the case when $A$ is the identity map. Namely, we show that every accumulation point of the primal sequence is a primal solution, and that, under certain technical assumptions, the dual sequence converges strongly to a dual solution (see Theorem \ref{dualconv}).

Our analysis opens the way for the application of DSG to challenging optimal control problems, which are infinite dimensional optimization problems. It is particularly interesting to apply the results of the present paper to the pure penalty method, i.e., when $A=0$. The latter case is not covered by the results in \cite{BKMinexact2010}.

\def\cprime{$'$}


\begin{thebibliography}{1}












\bibitem{Ber1999}
 D.P. Bertsekas,
\emph{Nonlinear Programming},
Athena Scientific, second edition, 1999.



\bibitem{Brezis}
 H. Brezis,
\emph{Functional Analysis, Sobolev Spaces and Partial Differential Equations},
Springer, Berlin, 2011.

\bibitem{Bur2011}
 R.S. Burachik, 
On primal convergence for augmented Lagrangian duality, 
\emph{Optimization}, \textbf{60}: 979--990 (2011).

\bibitem{B2017}
 R.S. Burachik,
On asymptotic Lagrangian duality for nonsmooth optimization,
\emph{ANZIAM journal}, \textbf{58}: 93--123 (2017).

\bibitem{BGKIdsg}
 R.S. Burachik, R.N. Gasimov, N.A. Ismayilova, and C.Y. Kaya,
On a modified subgradient algorithm for dual problems via sharp augmented Lagrangian,
\emph{J. Global Optim}, \textbf{34}: 55--78 (2006).


\bibitem{BI}
 R.S. Burachik and A.N. Iusem, 
\emph{Set-Valued Mappings and Enlargements of Monotone Operators},
Springer, Berlin, 2008.

\bibitem{BIMdsg}
 R.S. Burachik, A.N. Iusem, and J.G. Melo, 
 A Primal Dual Modified Subgradient Algorithm with Sharp Lagrangian,
\emph{J. Global Optim}, \textbf{46}: 55--78 (2010).

\bibitem{BIM2010}
 R.S. Burachik, A.N. Iusem, and J.G Melo,
Duality and exact penalization for general augmented Lagrangians,
\emph{J. Optim. Theory Appl}, \textbf{147}(1): 125--140 (2010).

\bibitem{BIMinexact2013}
 R.S. Burachik, A.N. Iusem and J.G. Melo, 
An inexact modified subgradient algorithm for primal--dual problems via Augmented Lagrangians,
\emph{J. Optim. Theory Appl}, \textbf{157}: 108--131 (2013).

\bibitem{BIM2015}
 R.S. Burachik, A.N. Iusem and J.G. Melo,  
The exact penalty map for nonsmooth and nonconvex optimization,
\emph{Optim.}, \textbf{64}(4): 717--738 (2015).

\bibitem{BFK2014}
 R.S. Burachik, W.P. Freire, and C.Y. Kaya,
Interior Epigraph Directions method for nonsmooth and nonconvex optimization via generalized augmented Lagrangian duality,
\emph{J. Global Optim.}: \textbf{60} (3), 501--529 (2014).


\bibitem{BKpenalty}
 R.S. Burachik and C.Y. Kaya,
An update rule and a convergence result for a penalty function method,
\emph{J. Ind. Manage. Optim.}, \textbf{3}(2): 381--403 (2007).

\bibitem{BKdsg}
 R.S. Burachik and C.Y. Kaya,
A deflected subgradient method using a general augmented Lagrangian duality with implications on penalty methods. In: R.S. Burachik, Yao, J.C. (eds.) \emph{Variational Analysis and Generalized Differentiation in Optimization and Control}, Springer Optimization and Its Applications,  \textbf{47}: 109--132. Springer, New York, (2010).




\bibitem{BKpenaltypara}
 R.S. Burachik and C.Y. Kaya,
An augmented penalty function method with penalty parameter updates for nonconvex optimization,
\emph{Nonlinear Anal. Theory Methods Appl.}, \textbf{75}(3): 1158--1167
(2012).


\bibitem{BKMinexact2010}
 R.S. Burachik, C.Y. Kaya and M. Mammadov,
An inexact modified subgradient algorithm for nonconvex optimization,
\emph{Comput. Optim. Appl}, \textbf{45}, 1--24 (2010).

\bibitem{BKP2021}
 R.S. Burachik, C.Y. Kaya, and C.J. Price,
A primal-dual penalty method via rounded weighted-$\ell_1$ Lagrangian duality,
\emph{Optim.}(2021).


\bibitem{BR2007}
 R.S. Burachik and A.M. Rubinov, 
Abstract convexity and augmented Lagrangians,
\emph{SIAM J. Optim.}, \textbf{18}: 413-436 (2007).

\bibitem{BY2011}
 R.S. Burachik and X.Q. Yang, 
Asymptotic strong duality, 
\emph{Numer. Algebra, Control Optim.}, \textbf{1}(3): 539--548 (2011).


\bibitem{BYZ2017}
 R.S. Burachik, X.Q. Yang and Y.Y. Zhou,
Existence of augmented Lagrange multipliers for semi-infinite programming problems, 
\emph{J. Optim. Theory Appl}, \textbf{173}: 471--503 (2017).


\bibitem{CC2010}
Y. Chen and M. Chen,
Extended duality for nonlinear programming,
\emph{Comput. Optim. Appl.}, \textbf{47}(1):33--59 (2010).









\bibitem{Dol2018}
 M.V. Dolgopolik,
Augmented Lagrangian functions for cone constrained optimization: the existence of global saddle points and exact penalty property,
\emph{J. Global Optim.}, \textbf{71}: 237--296 (2018).

\bibitem{Dol20182}
  M.V. Dolgopolik,
A Unified Approach to the Global Exactness of Penalty and Augmented Lagrangian Functions I: Parametric Exactness, 
\emph{J Optim. Theory Appl.}, \textbf{176}: 728--744 (2018).







\bibitem{Gdsg}
 R.N. Gasimov,
Augmented {L}agrangian duality and nondifferentiable,
optimization methods in nonconvex programming,
\emph{J. Global Optim.}, \textbf{24}:  187--203 (2002).




\bibitem{HY2003}
X.X. Huang and X. Q. Yang,
A unified augmented Lagrangian approach to duality and exact penalization, 
\emph{Math. Oper. Res}, \textbf{28}: 533--552 (2003).



\bibitem{HY2005}
X.X. Huang  and  X. Q. Yang,
Further study on augmented Lagrangian duality theory,
\emph{J. Global Optim.}, \textbf{31}: 193--210 (2005).














\bibitem{Kreyszig}
 E. Kreyszig,
\emph{Introductory Functional Analysis with Applications},
Wiley, 1978.



\bibitem{Li1997S}
  D. Li, 
Saddle-point generation in nonlinear nonconvex optimization, 
\emph{Nonlinear Anal.}, \textbf{30}: 4339--4344 (1997).













\bibitem{NO2008}
 A. Nedic, and A. Ozdaglar,
A geometric framework for nonconvex optimization duality using augmented Lagrangian functions, \emph{J. Global Optim.}, \textbf{40}: 545--573 (2008).

\bibitem{Nes2004}
 Y. Nesterov,
\emph{Introductory Lectures on Convex Optimization: A Basic Course},
Kluwer Academic Publishers, 2004.




\bibitem{Pol1997}
 B. Polyak,
\emph{Introduction to Optimization},
Optimization Software, Inc., 1987.






\bibitem{RockWets}
 R.T. Rockafellar and R.J.B. Wets,
\emph{Variational Analysis},
Springer-Verlag, Berlin, 1998.

\bibitem{Rub2000}
 A.M. Rubinov, 
\emph{Abstract Convexity and Global Optimization},
Kluwer Acadernic Publishers, 2000. 

\bibitem{RYlagrType}
 A.M. Rubinov and X. Yang,
\emph{Lagrange-type functions in constrained non-convex optimization},
Springer Science and Business Media, 2013.



















\bibitem{WYY2012}
 C.Y. Wang, X.Q. Yang and X.M. Yang,
Nonlinear augmented Lagrangian and duality theory, 
\emph{Math. Oper. Res}, \textbf{38}: 740--760 (2012).













\bibitem{ZY2008}
 L. Zhang, and  X. Yang,
An augmented Lagrangian approach with a variable transformation in nonlinear programming, 
\emph{Nonlinear Anal.}, \textbf{69}: 2095--2113 (2008).


\bibitem{ZY2004}
Y.Y. Zhou and X.Q. Yang,
Some results about duality and exact penalization, 
\emph{J. Global Optim.}, \textbf{29}: 497--509 (2004).

\bibitem{ZY2006}
Y.Y. Zhou  and X.Q. Yang,
Augmented Lagrangian function, nonquadratic growth condition and exact penalization, \emph{Oper. Res. Lett.}, \textbf{34}: 127--134 (2006).

\bibitem{ZY2009}
Y.Y. Zhou  and X.Q. Yang,
Duality and penalization in optimization via an augmented Lagrangian function with applications, 
\emph{J. Optim. Theory Appl.}, \textbf{140}: 171--188 (2009). 

\bibitem{ZY2012}
Y.Y. Zhou  and X.Q. Yang, 
Augmented Lagrangian functions for constrained optimization problems, 
\emph{J. Global Optim.}, \textbf{52}: 95--108 (2012).

\end{thebibliography}
\end{document}